\documentclass[10pt]{amsart}
\usepackage[all]{xy}
\pdfoutput=1
\usepackage{amsmath, amsthm, amssymb,slashed,stmaryrd}
\usepackage{enumitem}
\usepackage{microtype}
\usepackage{lmodern}
\usepackage{ifpdf}
\usepackage[pdftex]{graphicx}
\usepackage{tikz}
\usepackage{verbatim}
\usetikzlibrary{matrix,arrows,calc}
\usepackage{mathtools}
\usepackage{csquotes}

\usepackage{color}
\definecolor{dkgreen}{rgb}{0,0.5,0}
\usepackage{hyperref}
\hypersetup{
    linktoc=page,
	colorlinks = true,
	citecolor = dkgreen,
}

\usepackage{tikz-cd}
\usetikzlibrary{arrows, decorations.markings}
\usetikzlibrary{shapes.geometric}
\usepackage[capitalize]{cleveref}
\usepackage{bm}
\usepackage{accents}
\usepackage[super]{nth}
\usepackage{mathrsfs}

\usepackage{appendix}

\numberwithin{equation}{subsection}
\newtheorem{thm}{Theorem}[subsection]

\newtheorem{lem}[thm]{Lemma}
\newtheorem{prop}[thm]{Proposition}

\newtheorem{cor}[thm]{Corollary}

\theoremstyle{definition}

\newtheorem{example}[thm]{Example}
\newtheorem{defn}[thm]{Definition}
\newtheorem{rmk}[thm]{Remark}
\newtheorem{question}[thm]{Question}

\newtheorem{claim}[thm]{Claim}

\newtheorem*{conjecture*}{Conjecture}
\newtheorem*{theorem*}{Theorem}
\newtheorem*{claim*}{Claim}
\newtheorem*{corollary*}{Corollary}
\newtheorem*{notation*}{Notation}

\theoremstyle{definition}

\DeclareSymbolFont{bbold}{U}{bbold}{m}{n}
\DeclareSymbolFontAlphabet{\mathbbold}{bbold}

\def\Hom{{\sf Hom}}
\def\SL{{\rm SL}}
\def\PSL{{\rm PSL}}
\def\PGL{{\rm PGL}}

\def\Rep{{\rm Rep}}
\def\Spec{{\rm Spec}}

\def\Gal{{\rm Gal}}

\def\Sh{{\rm Sh}}

\def\Aut{{\rm Aut}}

\def\K3{{\rm K3}}

\def\GL{{\rm GL}}

\def\Res{{\rm Res}}

\def\ad{{\rm ad}}

\def\tr{{\rm tr}}
\def\ad{{\rm ad}}
\def\nrd{{\rm nrd}}
\def\exc{{\rm exc}}

\newcommand{\Shbo}{\Sh_{B, \mc{O}}}

\newcommand{\Qbar}{\overline{\mathbb{Q}}_{}}

\newcommand{\Qzmodrs}{\mathrm{Mod}_{\mathrm{rs}}(\Qbar(z)\langle\partial_z\rangle)}
\newcommand{\Conn}{\mathrm{Conn}}

\newcommand{\C}{\mathbb C}

\newcommand{\Mtil}{\widetilde{\mc{M}}}
\newcommand{\Ctil}{\widetilde{\mc{C}}}

\newcommand{\defeq}{\vcentcolon=}

\makeatletter
\newcommand{\colim@}[2]{%
  \vtop{\m@th\ialign{##\cr
    \hfil$#1\operator@font colim$\hfil\cr
    \noalign{\nointerlineskip\kern1.5\ex@}#2\cr
    \noalign{\nointerlineskip\kern-\ex@}\cr}}%
}
\newcommand{\colim}{%
  \mathop{\mathpalette\colim@{}}\nmlimits@
}
\makeatother

\newcommand\nc{\newcommand}
\nc\mf\mathfrak
\nc\mc\mathcal
\nc\mb\mathbb
\nc\msf\mathsf
\nc\mscr\mathscr
\nc\mbf\mathbf

\def\resp{\text{resp.\kern.3em}}
\def\eg{\textit{e.g.}\kern.3em}
\def\ie{\textit{i.e.}\kern.3em}
\def\loccit{loc.\kern3pt cit.{}\xspace}

\begin{document}

\title[On Siegel's problem and Dwork's conjecture for $G$-functions]{On Siegel's problem and Dwork's conjecture \\ for $G$-functions}

\author{Javier Fres\'an}
\address{Sorbonne Université and Université Paris Cité, CNRS, IMJ-PRG, 75005 Paris}
\email{javier.fresan@imj-prg.fr}

\author{Yeuk Hay Joshua Lam}
\address{Humboldt Universität zu Berlin,
       Institut für Mathematik--Alg.Geo.,
        Rudower Chaussee~25
        Berlin, Germany}
\email{joshua.lam@hu-berlin.de}

\author{Yichen Qin}
\address{Humboldt Universität zu Berlin,
       Institut für Mathematik--Alg.Geo.,
        Rudower Chaussee~25
        Berlin, Germany}
\email{yichen.qin@hu-berlin.de}

\date{\today}

\begin{abstract}  
We answer in the negative Siegel's problem for $G$-functions, as formulated by Fischler and Rivoal. Roughly, we prove that there are $G$\nobreakdash-functions that cannot be written as polynomial expressions in algebraic pullbacks of hypergeometric functions; our examples satisfy differential equations of order two, which is the smallest possible. In fact, we construct infinitely many non\nobreakdash-equivalent rank-two local systems of geometric origin which are not algebraic pullbacks of hypergeometric local systems, thereby providing further counterexamples to Dwork's conjecture and answering a question by Krammer. The main ingredients of the proof are a Lie algebra version of Goursat's lemma, the monodromy computations of hypergeometric local systems due to Beukers and Heckman, as well as results on invariant trace fields of Fuchsian groups. 
\end{abstract}

\maketitle 
\setcounter{tocdepth}{1}
\tableofcontents

\section{Introduction}

\newcommand*{\ldblbrace}{\{\mskip-5mu\{}
\newcommand*{\rdblbrace}{\}\mskip-5mu\}}

In his 1929 memoir on applications of Diophantine approximation, Siegel \cite{siegel2014einige} introduced two classes of functions of a single variable $z$ that he named $G$-functions and $E$-functions, as they were reminiscent of the geometric and the exponential series, respectively. The present work focuses on $G$-functions: roughly, these are power series $F(z)=\sum_{n\geq 0} a_n z^n$ with algebraic coefficients $a_n\in \Qbar$ such that 
\begin{enumerate}
    \item $F(z)$  converges in  a neighborhood of $0$ and satisfies a linear differential equation with coefficients in $\Qbar(z)$; 
    \item  the sequence $d_n$ of ``common denominators'' of $a_0, \ldots, a_n$ grows at most exponentially with $n$. 
\end{enumerate}
Siegel gave many examples of $G$-functions, with the primary ones being the so-called hypergeometric functions, defined explicitly as
\begin{equation}\label{eqn:explicithyp}
 _{p+1}F_{p}[a_1, \dots , a_{p+1}; b_1, \dots , b_p\,|\,z] \defeq \sum_{n=0}^\infty \frac{(a_1)_n\cdots (a_{p+1})_n}{(b_1)_n\cdots (b_p)_n} \frac{z^n}{n!},
\end{equation}
for parameters $a_i \in \mb{Q}$ and $ b_j \in \mb{Q} \setminus \mb{Z}_{\leq 0}$, where $(x)_n$ is the Pochhammer symbol, given by $(x)_0=1$ and $(x)_n= x(x+1)\cdots (x+n-1)$ for $n\geq 1$. For example, $_{1}F_0[1;\emptyset|z]=\sum_{n\geq 0} z^n$ is the usual geometric series.

Hypergeometric functions have a rich history. For $p=1$, their coefficients seem to have been first studied by Wallis in 1656 \cite[Scholium to Prop.\,190]{wallis}, who coined the term  \emph{hypergeometric progression} in contrast to the usual geometric progressions. These functions were then studied by Stirling, Euler, Gauss, Riemann, and many others; we refer the reader to \cite{dutka1984early} for the early history of hypergeometric functions. More recently, they were revisited by Katz in his influential works~\cite{Katz1990,Katz96rigid} on differential equations and rigid local systems on a punctured projective line. 

Conjecturally, $G$-functions are precisely the elements of $\Qbar[\![z]\!]$ satisfying differential equations \emph{of geometric origin}, i.e., products of irreducible factors of Picard\nobreakdash--Fuchs differential equations arising from one-parameter families of algebraic varieties over~$\Qbar$; see \cite[Intro.\,Conj.]{andre-g-functions}. That solutions of Picard\nobreakdash--Fuchs differential equations are indeed $G$\nobreakdash-functions was proved by Andr\'e \cite[Intro.\,Thm.\,B]{andre-g-functions}, and the converse is the wide open Bombieri--Dwork conjecture. For this reason, $G$-functions have close ties to arithmetic geometry \cite{andre25}. The main result of this paper is a concrete statement about $G$\nobreakdash-functions: we show the existence of a $G$-function that cannot be expressed as a polynomial  in algebraic pullbacks of hypergeometric functions.

\subsection{Main results}
The set of $G$-functions is closed under addition and multiplication, as well as algebraic substitutions of the variable. Therefore, starting from the hypergeometric $G$-functions \eqref{eqn:explicithyp}, one can produce many new~$G$\nobreakdash-functions in this way. The analog of Siegel's problem for $E$-functions, as stated by Fischler and Rivoal in~\cite[Question 2]{fischlerrivoal}, asks whether this produces all $G$-functions: 

\begin{question}[Siegel's problem for $G$-functions]\label{siegelproblem}
Is it possible to write every~$G$\nobreakdash-function as a polynomial with coefficients in $\Qbar$ in functions of the form 
\begin{equation}\label{eqn:hypgfunction}
    \mu(z)\cdot \ _{p+1}F_{p}[a_1, \dots , a_{p+1}; b_1, \dots , b_p\,|\, \lambda(z)],
\end{equation}
for integers $p\geq 0$, parameters $a_i \in \mb{Q}$ and $b_j \in \mb{Q} \setminus \mb{Z}$, and algebraic functions $\mu(z)$ and $ \lambda(z)$ such that $\lambda(z)$ is holomorphic at $z=0$, and  $\lambda(0)=0$?
\end{question}

The reason for the conditions on $\lambda$ is that $G$-functions are, by definition, centered at zero. One can also work with a slightly more general notion that accommodates poles and log terms, but we will ignore this subtlety in the introduction. 

\begin{rmk}\label{rmk:FJ-vs-us}
This is almost the exact analog of Siegel's original problem for~$E$\nobreakdash-functions, which was only recently answered in the negative by Fres\'an and Jossen~\cite{fresanjossen}. The main difference is that there is much more flexibility in the case of $G$-functions: while composing $G$-functions with algebraic functions $\lambda(z)$ as above gives again~$G$\nobreakdash-functions, the same is seldom true for $E$-functions. 
\end{rmk}

Fischler and Rivoal showed that \cref{siegelproblem} likely had a negative answer. In fact, combining \cite[Thm.\,7.2]{fischlerrivoal} with the argument on page 90 of \textit{loc.\,cit.} due to André, the exponential Grothendieck period conjecture from \cite[Conj.\,1.3.2]{exp-mot} implies that there are counterexamples to a weaker version of \cref{siegelproblem}, assuming further that the algebraic functions $\lambda$ have a common singularity at which they all tend to $\infty$. 
Our main result provides unconditional and explicit counterexamples. We say that a $G$-function is \textit{of differential order $k$} if $k$ is the minimal order of a non-zero differential operator with coefficients in $\Qbar(z)$ annihilating~it.

\begin{thm}\label{thm:gfunction}
    The answer to \cref{siegelproblem} is negative. More precisely, there exist $G$-functions of differential order $2$ that cannot be written as a $\Qbar$-polynomial expression in functions of the form \eqref{eqn:hypgfunction}. 
\end{thm}

In fact, we prove that there are infinitely many such $G$-functions of differential order 2, even up to reparametrizations of the coordinate $z$; this is \cref{thm: infinitely-many-examples}, which answers a question of Daan Krammer's.

As an explicit example, we show in \cref{prop:Krammer-example} that the series expansion at~$0$ of all solutions in $\Qbar[\![z]\!]$ to the differential equation  
\begin{equation}\label{eq:counter-example}
     P(z)F''(z)+\frac{P'(z)}{2}F'+\frac{z-10}{18}F=0,
\end{equation}
where $P(z)=(z-1)(z-2)(z-82)$, are counterexamples. In particular, one such solution is the $G$-function 
\[
F(z) = 1-\frac{5}{2952}z^2-\frac{889}{726192}z^3-\frac{1851985}{2143718784}z^4-\frac{110984489}{175784940288}z^5-\cdots. 
\]

\begin{rmk} The André--Chunovsky--Katz theorem that the minimal differential equation of a $G$-function has regular singularities and rational exponents, as summarized in \cite[\S3]{Andre-Gevrey-I}, implies that every $G$-function of differential order~$1$ is algebraic, so the counterexamples to \cref{siegelproblem} in the above theorem have the smallest possible order. One of the main points of the present work is to provide examples of minimal order, which may moreover be explicitly computed as above.
\end{rmk}

\subsection{A sketch of proof}\label{section:proofsketch} 
Our strategy to prove \cref{thm:gfunction} begins the same way as that of Fres\'an and Jossen \cite{fresanjossen} in the case of $E$-functions, in that we reformulate the problem in terms of Tannakian categories. More precisely, we want to construct two Tannakian categories~$\mbf{H} \subset \mbf{G}$ and reduce to proving that this inclusion is strict. 

Following Andr\'e \cite{andre-g-functions}, we define $\mbf{G}$ as the category of $\Qbar(z)\langle\partial_z\rangle$-modules arising from algebraic geometry in \cref{defn:G}. The construction of~$\mathbf{H}$ is slightly trickier than that from \cite{fresanjossen}, which is simply generated by the differential equations of hypergeometric $E$-functions, since we must consider $G$-functions of the form~$F(\lambda(z))$ where $F$ are hypergeometric functions and $\lambda(z)$ algebraic functions. Instead, we work with the Tannakian category $\mbf{H}$ in \cref{defn:H} generated by~$\Qbar(z)\langle\partial_z\rangle$\nobreakdash-modules of the form $\pi_{2*}\pi_1^*\mc{H}$, built from hypergeometric connections with regular singularities~$\mc{H}$ using correspondences
    \begin{equation}\label{eqn:correspintro}
    \begin{tikzcd}[column sep=small]
        & C \arrow[dl, swap, "\pi_1"] \arrow[dr, "\pi_2"] & \\
        \mb{P}^1_{\lambda} & & \mb{P}^1_{z},
    \end{tikzcd}
    \end{equation}
where $C$ is a smooth projective curve and $\pi_1$ and $\pi_2$ are finite Galois covers. We deduce from a theorem by Katz \cite[\S5.4]{Katz1990} that $\mbf{H}$ is a full subcategory of $\mbf{G}$.

At this point, our strategy for showing that $\mbf{H}$ is not the whole $\mbf{G}$  diverges from that of Fres\'an and Jossen. Roughly, they prove that the similar inclusion $\mbf{H}\subset \mbf{E}$ is strict by an ingenious analysis of the geometry of the singularities of the \emph{Fourier transforms} of objects of their category $\mbf{H}$. Because of the algebraic pullbacks in our definition of~$\mbf{H}$, the singularities tell us almost nothing in the present situation. 

Our strategy is then to find some other invariant of objects of $\mbf{H}$. We do so by analyzing the \emph{trace fields} of objects of $\mbf{H}$ (more precisely, their \emph{adjoint trace fields}, which have better properties; this difference will be suppressed in the introduction). Under the Riemann--Hilbert correspondence, an object of $\mbf{H}$ corresponds to a local system on $\mb{P}^1\setminus S$ for some finite subset $S \subset \mathbb{P}^1(\Qbar)$, and we may look at the field generated by traces of elements of~$\pi_1(\mb{P}^1\setminus S)$ acting on the associated representation. 

We first observe that the work of Katz on rigid local systems easily implies that the trace field of a hypergeometric connection $\mc{H}$ is abelian (\cref{prop:trace-field-hyp}). Naively, one may think (as we did initially) that it is enough to find a rank-two connection $\mc{V}$ of geometric origin on~$\mb{P}^1\setminus S$ whose trace field is not abelian. This does not work, since the objects of~$\mbf{H}$ include arbitrary summands of tensor products of hypergeometric local systems, and it is not clear at all that the property of having abelian trace field is invariant under taking summands. 

We get around this by applying a Lie algebra version of Goursat's lemma due to Fres\'an and Jossen, and deduce that if a simple object $\mc{V}$ of $\mbf{G}$ belongs to $\mbf{H}$, then it is Lie-generated by a single hypergeometric connection $\mc{H}$, under a correspondence as in~\eqref{eqn:correspintro}. We then apply the computation by Beukers and Heckman (quoted as \cref{thm:beukersheckman} below) of the differential Galois groups of all regular-singular hypergeometric connections; this allows us to deduce that, if $\mc{V}$ is of rank two with the maximal possible differential Galois group, then the same must be true of $\mc{H}$. 

From here, we can apply the trace field argument as outlined above using certain techniques related to the adjoint trace fields of rank-two local systems and Fuchsian groups (\cref{lemma:comminvariance}). Therefore, to conclude the proof of \cref{thm:gfunction}, it suffices to find such a $\mc{V}$ with non-abelian trace field. Indeed, one can find many examples coming from rational Shimura curves.

\subsection{Previous work}
Siegel's problem for $G$-functions is closely related to a conjecture by Dwork \cite[Conj.~7.4]{dwork1990differential}, which asked whether every $G$-function of order~2 with infinite monodromy may be written as 
\[
\mu_1(z)\cdot _2F_1[a_1, a_2; b_1\,|\, \lambda(z)]+\mu_2(z) ,
\]
for some algebraic functions $\lambda(z), \mu_1(z), \mu_2(z)$.  A counterexample to Dwork's conjecture was first found by Krammer \cite{krammer} by considering a Shimura curve and crucially using   Takeuchi's classification \cite[Table~1]{takeuchi1977commensurability} of arithmetic triangle groups. Several more counterexamples were then given by Bouw and M\"oller \cite{bouwmoller}, coming from rational Teichm\"uller curves. Our trace field criterion provides yet more counterexamples to Dwork's conjecture. With our techniques,  we can also prove that the examples of Krammer and Bouw and M\"oller negatively answer  \cref{siegelproblem}.

 As already pointed out by Krammer \cite[\S12]{krammer} in 1996,  his work provides only finitely many counterexamples, up to reparametrizations of~$\mb{P}^1$,  to Dwork's conjecture; the same is true of that of Bouw and M\"oller. In \cref{sec:infinite}, we show that there are infinitely many ``non-hypergeometric'' $G$-functions (even up to reparametrizations of~$\mb{P}^1$), thereby answering the question raised by Krammer. These are provided by local systems studied by Lam and Litt \cite{lamlitt}, and applying the Andr\'e--Pink--Zannier conjecture in certain cases that were proven by Richard and~Yafaev \cite{richard2021generalised}. 
 In a recent work \cite{vvZ}, van Hoeij, van Straten, and Zudilin computed explicitly the $G$-function corresponding to one of these local systems, and also suggested that these give new counterexamples to  Dwork's conjecture. 

As mentioned above, Fischler and Rivoal \cite{fischlerrivoal} provided a negative answer to a weaker version of \cref{siegelproblem}, assuming the Grothendieck period conjecture. While the present paper was in preparation, Dreyfus and Rivoal raised a generalization of \cref{siegelproblem} in \cite[Question 3]{Rivoal-new-G-function}, and gave negative answers to weaker versions of both the original and the generalized question. More precisely, for each integer $n$ they construct a $G$-function that is not a polynomial expression with coefficients in $\overline{\mb{C}(x)}$ of solutions of differential operators with $n$ singularities, and from this they deduce counterexamples under the additional assumption that the~$\lambda$'s are rational functions with denominators of uniformly bounded degree. 

\subsection{Organization} This paper is organized as follows. In \cref{sec:recollections}, we collect some basic properties of $G$-functions, trace fields, and hypergeometric connections. In \cref{sec:reformulation}, we construct the Tannakian categories~$\mathbf{H}\subset \mathbf{G}$ and reduce \cref{siegelproblem} to demonstrating that this inclusion is strict. Using the trace field argument, \cref{thm:gfunction} is proved in \cref{sec:proof}, where we provide counterexamples derived from Shimura and Teichmüller curves. Finally, in \cref{sec:infinite}, we construct infinitely many non-hypergeometric $G$-functions of order $2$ using the generalized Andr\'e\nobreakdash--Pink\nobreakdash--Zannier conjecture for Shimura varieties of abelian type from \cite{richard2021generalised}. Given that this paper relies on results and techniques from various areas, we tried to include enough background material to make it as self-contained as~possible.

\subsection{Notation}
Throughout, $\Qbar$ denotes the algebraic closure of $\mathbb{Q}$ in $\mathbb{C}$. A number field is called \textit{abelian} if its Galois closure has abelian Galois group. 

We use calligraphic letters such as $\mc{V}$ to denote the vector bundle underlying a connection, and bold letters such as $\mb{V}$ to denote local systems. 

For a morphism $f\colon X\to Y$ of smooth complex varieties, we use $f^*$ and $f_*$ for the pushforward and pullback of $\mathscr{D}$-modules, as in \cite[\S 5.1]{Katz1990}. For $\mbf{T}$ a Tannakian category and an object $M$ in $\mbf{T}$, we denote by $\langle M\rangle ^{\otimes}$ the Tannakian subcategory of~$\mbf{T}$ generated by $M$.

\subsection*{Acknowledgements}  
 This project was initiated after a talk by the first author (and attended by the second and third authors) at the workshop \emph{Hodge theory, tropical geometry, and o-minimality} organized by Omid Amini and Bruno Klingler in October~2023 in Berlin; we thank them heartily for providing a fantastic working environment during this workshop. We are grateful to Alin Bostan, Thomas Krämer, Raju Krishnamoorthy, and Daniel Litt for enlightening discussions.

Fresán was partially supported by the grant ANR-23-CE40-0011 of the Agence Nationale de la Recherche. Lam was supported by a Dirichlet Fellowship and the DFG Walter Benjamin grant \emph{Dynamics on character varieties and Hodge theory} (project number: 541272769). Qin was supported by the European Research Council (ERC) under the European Union’s Horizon 2020
research and innovation program (grant agreement no.\,101020009, project TameHodge).

\section{Preliminaries on hypergeometric connections}\label{sec:recollections}

After briefly recalling the definitions of $G$-functions and trace fields, we gather in Proposition \ref{prop:trace-field-hyp} and Theorem \ref{thm:beukersheckman} the properties of hypergeometric connections and their associated local systems on which the proof of the main theorem relies. 

\subsection{\texorpdfstring{$G$}{G}-functions and \texorpdfstring{$G$}{G}-operators}

\begin{defn}\label{defn:g-function}
    A \emph{$G$-function} is a formal power series with algebraic coefficients
    \[
   G(z)=\sum_{n=0}^{\infty} a_n z^n \in \Qbar[\![z]\!] 
    \]
    that is annihilated by a non-zero differential operator $L\in \Qbar[z]\langle\partial_z\rangle$ and satisfies the following growth conditions: for each $n\geq 1$, writing $d_n\geq 1$ for the smallest integer such that~$d_na_1, \dots , d_na_n$ are algebraic integers, there exists a real number $C>0$ such that $|\sigma(a_n)|\leq C^n$ and $d_n\leq C^n$ hold for all $\sigma \in \Gal(\Qbar/\mb{Q})$ and all $n \geq 1$.
    \end{defn}
 
The set of $G$-functions is closed under addition and multiplication; this is more or less straightforward from the definition and was already observed in \cite{siegel2014einige}. It is also closed under the Hadamard product 
\[
(F\star G)(z) := \sum_{n=0}^{\infty} a_nb_nz^n
\] for $F(z)=\sum_{n=0}^\infty a_nz^n$ and $G(z)=\sum_{n=0}^\infty b_nz^n$, as proved in {\cite[Thm.\,D]{andre-g-functions}}. 

\begin{defn}\label{def:functions-with-monodromy}
Let $\mb{C}\ldblbrace z\rdblbrace$ denote the subring of $\mb{C}[\![z]\!]$ consisting of power series with positive radius of convergence. Write $\mc{M}$ for the differential algebra of \emph{convergent Laurent series with monodromy}  \cite[\S1.1]{fresanjossen}, defined as 
\[
\mc{M}:= \mb{C}\ldblbrace z\rdblbrace [(z^a)_{a\in \mb{C}}, \log(z)].
\]
We denote by $\mc{G}$ the differential subalgebra of $\mc{M}$ consisting of complex linear combinations of elements of the form 
\[
z^{a_i} \log(z)^{b_i} F_i(z), 
\]
for rational numbers $a_i$, integers $b_i\geq 0$, and $G$-functions~$F_i$.  
\end{defn}

\begin{defn} A differential operator $L \in \mathbb{Q}(z)\langle\partial_z\rangle$ is called a \textit{$G$-operator} if $L$ admits a basis of solutions in the algebra $\mathcal{G}$. 
\end{defn}

See \cite[\S1.7]{fresanjossen} for the equivalence of this definition with the more Diophantine one in terms of Galochkin's condition which is usually found in the literature. 

\subsection{Trace fields and adjoint trace fields}

\begin{defn}\label{defn:trace-fields}
    Let $\Gamma$ be a finitely generated group and $\rho\colon \Gamma \rightarrow \GL_n(\mb{C})$ a complex representation of $\Gamma$.
    \begin{enumerate} 
    \item For any automorphism $\sigma \in \Aut(\mb{C})$, let $\sigma(\rho)$ denote the composition 
    \[
    \Gamma \stackrel{\rho}{\longrightarrow} \GL_n(\mb{C}) \stackrel{\sigma}{\longrightarrow} \GL_n(\mb{C}).
    \] The \emph{field of definition} of $\rho$ is $k:=\mb{C}^{G}$, where $G \subset \Aut(\mb{C})$ denotes the subgroup defined as $G:=\{\sigma \in \Aut(\mb{C})\,|\,\sigma(\rho)\simeq \rho\}$. 
    
    \item The  \emph{trace field} of $\rho$ is the  subfield of $\mb{C}$ generated by the traces of $\rho(\gamma)$, \ie 
    \[
    \mb{Q}\big(\{\tr\, \rho(\gamma)\,|\,\gamma \in \Gamma\}\big)\subset \mb{C}. 
    \] 
    We emphasize that, in our setup, the trace field is naturally a subfield of~$\mb{C}$; in what follows, when we say that two trace fields are equal we mean that they are the same subfield of $\mb{C}$, and not just abstractly isomorphic.
    
    \item The \emph{adjoint trace field} of $\rho$ is the trace field of the adjoint representation \[
    \ad(\rho)\colon  \Gamma\rightarrow \GL(\mf{gl}_n).\]
    Note that as a representation of $\Gamma$, $\ad(\rho)$ decomposes as $\mathrm{triv}\oplus \ad^0(\rho)$, where $\rm{triv}$ denotes the trivial representation and $\ad^0(\rho)$ the representation of $\Gamma$ on~$\mf{sl}_n\subset \mf{gl}_n$, the subspace of traceless matrices. Then the adjoint trace field of~$\rho$ is also equal to the trace field of $\ad^0(\rho)$.
    \end{enumerate}
\end{defn}

\begin{rmk}\label{rem:adjointandtracefields}
    The adjoint trace field is contained in the trace field. Indeed, let $\gamma \in \Gamma$ and suppose that $\rho(\gamma)$ has generalized eigenvalues $\alpha_1, \dots , \alpha_n$. Then the endomorphism $\ad(\rho)(\gamma)$ has trace 
    \[
    \sum_{1 \leq i,j \leq n}\alpha_i\alpha^{-1}_j = (\alpha_1+\cdots +\alpha_n)(\alpha_1^{-1}+\cdots +\alpha_n^{-1})=\tr\, \rho(\gamma)\cdot \tr\,\rho(\gamma^{-1}).
    \]
\end{rmk}

\begin{rmk}
We highlight that, for $\Gamma$ and $\rho$ as above, there is another natural field which is the minimal field $\mb{K}\subset \mb{C}$ such that  $\rho$ has a model over $\mb{K}$. This in general strictly contains the trace field of $\rho$---there is a Brauer obstruction class.  There seems to be some discrepancies in the literature about the naming of these fields: we refer the reader to the discussion in \cite{fields_of_definition}. In any case, we will not use the auxiliary field $\mb{K}$ in this paper, and hope that no confusions will be caused. 
\end{rmk}
\begin{lem}\label{prop:field of defn}
    The field of definition of a semisimple representation $\rho$ coincides with its trace field.
\end{lem}

\begin{proof}
Let $K$ and $F$ be the trace field and field of definition of $\rho$, respectively. We keep the notation $G\subset \Aut(\mb{C})$ from \cref{defn:trace-fields}. Since $\sigma(\rho)\simeq \rho$ holds for $\sigma\in G$, we deduce $\tr \, \rho(\gamma) = \tr \, \sigma(\rho)(\gamma)=\sigma (\tr \, \rho(\gamma))$, and hence the inclusion $K\subset F$. 

    On the other hand, recall from \cite[\S 20 6, VIII.376, Cor.\,(a)]{bourbakialg8} that two semisimple representations over a field of characteristic zero are isomorphic if and only if they have the same characters. This implies that $\tau(\rho)$ is isomorphic to $\rho$ for all $\tau\in \Aut(\mb{C}/K)$, and hence $\Aut(\mb{C}/K)\subset G$. Therefore, $K=F$, as required.
\end{proof}

\begin{lem}[Maclachan--Reid]\label{lemma:comminvariance}
Let $\Gamma$ be a finitely generated subgroup of $\SL_2(\mb{C})$, and $\rho\colon \Gamma\xhookrightarrow{} \SL_2(\mb{C})$ the tautological representation. Suppose $\rho$ is Zariski dense and not unitary. Then, the adjoint trace field of $\rho$ is an invariant of the commensurability class of $\Gamma$, \ie it is invariant under passing to finite index subgroups.
\end{lem}

\begin{proof}
   This is essentially \cite[Lem.\,3.5.6]{maclachlan2003arithmetic}, as we now explicate. Our group $\Gamma$ is non-elementary \cite[Def.\,1.2.1]{maclachlan2003arithmetic}, \ie its image in $\mathrm{PSL}_2(\mb{C})$  does not have a finite orbit in its action on the extended hyperbolic $3$-space $\mathbb{H}^3\cup \hat{\mathbb{C}}$. Indeed, this amounts to saying that no finite index subgroup of $\Gamma$ fixes a point in $\mb{H}^3\cup \hat{\C}$.  For the sake of contradiction, suppose that a finite index subgroup $\Gamma'\subset \Gamma$ fixes a point~$x\in \mb{H}^3\cup \hat{\C}$. Recall that, for the action of~$\mathrm{SL}_2(\mb{C})$ on $\mb{H}^3$ and $\hat{\C}$,  the stabilizer of $x\in \mb{H}^3$ is~$\mathrm{SU}(2)$, and the stabilizer of~$x\in \hat{\mb{C}}$ is a Borel subgroup. The first case implies that~$\rho$ is unitary, and the second implies that $\Gamma$ is not Zariski dense. As both of these contradict the assumptions, $\Gamma$ is non-elementary. 
   
   It then follows from  \cite[Thm.\,3.3.4, Lem.\,3.5.6]{maclachlan2003arithmetic} that $\mb{Q}(\{\tr(\gamma^2)\,|\,\gamma \in \Gamma\})$ is an invariant of the commensurability class, which is called the \emph{invariant trace field} in~\cite[Def.\,3.3.6]{maclachlan2003arithmetic}. But for a subgroup of $\SL_2(\mb{C})$, the adjoint trace field of $\rho$ is precisely this invariant trace field because of the relation $\tr(\gamma^2)=-2+\tr(\mathrm{Ad}(\gamma))$.
\end{proof}

\subsection{Hypergeometric connections}\label{sec:hypergeometric}

We briefly review some facts and conventions regarding hypergeometric connections. Let $n \geq 1$ be an integer and $a, b \in \mathbb{Q}^n$. The \emph{hypergeometric differential operator} with parameters $a$ and $b$ is 
\[
    \mathrm{Hyp}(a;b):=\prod_{i=1}^{n} (z\partial_z+b_i-1)-  z \prod_{j=1}^{n} (z\partial_z+a_j). 
\]
The hypergeometric functions $_{p+1}F_p[a_1, \dots , a_{p+1}; b_1, \dots , b_p\,|\,z]$ from \eqref{eqn:explicithyp} in the introduction are annihilated by the operators $\mathrm{Hyp}(a_1,\dots,a_{p+1};b_1,\dots,b_p,1)$. Thanks to the closure under Hadamard products, to show that  hypergeometric functions are~$G$\nobreakdash-functions, it suffices to show that series of the form 
\[
\sum_{n=1}^\infty \frac{(a)_n}{(b)_n}z^n
\]
are $G$-functions for $a \in \mb{Q}$ and $b \in \mb{Q} \setminus \mb{Z}_{\leq 0}$. Following Siegel \cite[Satz 5]{siegel2014einige}, this reduces to showing that the least common multiple of $1, 2, \dots , n$ is bounded by $C^n$ for some constant~$C$, which is a weak form of the prime number theorem.

Associated with  the hypergeometric differential operator $\mathrm{Hyp}(a;b)$ there is a connection $\mathcal{H}(a;b)$ of rank $n$ on $\mathbb{P}^1\backslash\{0,1,\infty\}$, called a \emph{hypergeometric connection}. It has regular singularities at $0,1,\infty$, and  is irreducible if and only if $a_i-b_j \not\in\mathbb{Z}$ for all indices $i$ and $j$ by \cite[Cor.\,3.2.1]{Katz1990}; by the Riemann--Hilbert correspondence, there is an associated local system on $\mb{P}^1\setminus \{0,1,\infty\}$, which we denote by $\mb{H}(a;b)$.

Assume now that $n=p+1$ for $p \geq 0$ and $b_{p+1}=1$. We set $\alpha_i=\exp(2\pi\sqrt{-1}a_i)$ and $\beta_i=\exp(2\pi\sqrt{-1}b_i)$, draw based loops around $0,1, \infty$ as in \cite[p.\,328]{beukersheckman}, and denote by $g_0, g_1, g_{\infty}$ the monodromy matrices of $\mb{H}(a;b)$ along these loops. (Notice that the roles of $\alpha_i,\beta_j$ and $a_i,b_j$ are interchanged compared to the convention  in~\cite{beukersheckman}.) 
 
\begin{prop}[{\cite[Prop.\,3.2, Thm.\,3.5]{beukersheckman}}]\label{prop:hyp local mono}
    The elements  $g_0, g_1, g_{\infty}$ satisfy
    \begin{align*}
        det(t-g_{\infty}) &= \prod_{i=1}^n (t-\alpha_i),\\
         det(t-g_{0}^{-1}) &= \prod_{i=1}^n (t-\beta_i),\\
         \mathrm{rk}(g_1-\mathrm{Id})&=1, \ \det(g_1)=\exp(2\pi \sqrt{-1}\sum_{i=1}^n (b_i-a_i)).
    \end{align*}
Moreover, for each generalized eigenvalue of $g_{\infty}$, there is a unique Jordan block with this generalized eigenvalue, and the same is true of $g_0$.       
\end{prop}

We prove a result on the adjoint trace fields of rank-two hypergeometric local systems for later use. We write $\mb{H}(a_1, a_2; b_1,1)$ for the rank-two local system \hbox{on~$\mb{P}^1\setminus \{0, 1, \infty\}$} corresponding to the connection $\mc{H}(a_1, a_2; b_1,1)$. 
\begin{prop}\label{prop:trace-field-hyp}
     The adjoint trace field $F$ of an irreducible hypergeometric connection $\mb{H}(a_1, a_2; b_1,1)$ satisfies:
    \begin{enumerate}
        \item $F$ is contained in  a cyclotomic field; 
        \item $F$ cannot be $\mb{Q}(\sqrt{D})$ for any odd squarefree integer $D\geq 7$. 
    \end{enumerate}
\end{prop}
\begin{proof}

The first part follows from Katz's theorem on rigid local systems, as we now explain. By \cref{prop:field of defn}, the trace field coincides with the field of definition. In view of \cref{rem:adjointandtracefields}, it then suffices to prove that the field of definition is contained in a cyclotomic field. For brevity, let us write $\mb{H}$ for $\mb{H}(a_1, a_2; b_1,1)$. Since the $a_i$ and $b_j$ are rational numbers, the eigenvalues of the local monodromies of $\mb{H}$ around~$0,1 , \infty$ are roots of unity, and hence there exists a torsion rank-one local system $\chi$ such that $\mb{H}':=\mb{H}\otimes \chi$ has trivial determinant. Since $\mb{H}$ has traces in a cyclotomic field if and only if $\mb{H}'$ does, we are reduced to proving the statement for~$\mb{H}'$. Let $F\subset \mb{C}$ be the field generated by the eigenvalues of the local monodromies of $\mb{H}'$ around~$0, 1, \infty$, which is evidently contained in a cyclotomic field.  We claim that the field of definition of $\mb{H}'$ is contained in (and in fact equal to) $F$. Indeed, it suffices to show that $\Aut(\mb{C}/F)$ is contained in  
\[
G':=\{\sigma \in \Aut(\mb{C})\,|\,\sigma(\mb{H}')\simeq \mb{H}'\}.
\]
The key point is that $\mb{H}'$ is \emph{physically rigid} \cite[Thm.~3.5.4]{Katz1990}, \ie it is determined by its local monodromies. From this, we may now conclude: any $\sigma \in \Aut(\mb{C}/F)$ fixes the conjugacy classes of local monodromies of $\mb{H}'$, and hence $\mb{H}'$ itself.

We now prove the second part. Suppose that the adjoint trace field of $\mb{H}$ is $\mb{Q}(\sqrt{D})$ for some odd squarefree integer $D\geq 7$. Write $\lambda=\alpha_1/\alpha_2$, $\mu=\beta_1$, $\nu= \beta_1/(\alpha_1\alpha_2)$, which are roots of unity of orders $\ell, m, r$, respectively. Then  the traces of $g_{\infty}, g_0, g_{1}$ on the adjoint representation are given by 
   \[
   2+\lambda+\lambda^{-1}, \ 2+ \mu+\mu^{-1}, \ 2+ \nu +\nu^{-1},
   \]
   respectively. Now $\lambda+\lambda^{-1}$ has degree $\phi(\ell)/2$ over $\mb{Q}$, and similarly for $\mu+\mu^{-1}$ and~$\nu+\nu^{-1}$. The assumption $D \geq 7$ implies that $\mb{Q}(\sqrt{D})$ is different from $\mb{Q}(\zeta+\zeta^{-1})$ for \emph{any} root of unity $\zeta$, so $\lambda+\lambda^{-1}, \mu+\mu^{-1}, \nu+\nu^{-1}$ must all be rational. In other words, each of the integers $\ell, m, r$ is one of $1,2,3,4,6$. 

   As above, we twist $\mb{H}$ by a character to obtain a local system $\mb{H}'$ with trivial determinant. Doing so, the local monodromies at $\infty, 0, 1$ have eigenvalues
   \begin{equation}\label{eqn:roots}
   \{\alpha_1\beta_1^{1/2}\nu^{1/2}, \alpha_2\beta_1^{-1/2}\nu^{1/2}\}, \{\beta_1^{1/2}, \beta_1^{-1/2}\}, \{\nu^{1/2}, \nu^{-1/2}\}
   \end{equation}
   respectively, for some choices of square roots $\beta_1^{1/2}, \nu^{1/2}$. Note that all the roots of unity appearing in  \eqref{eqn:roots} have orders dividing $24$. Since twisting by a character has no effect on the adjoint trace field, we must have that $\mb{Q}(\sqrt{D})$ is contained in the trace field of $\mb{H}'$; by the above computation, we deduce that $\mb{Q}(\sqrt{D})$ is contained in 
   \[\mb{Q}(\zeta_{24}+\zeta_{24}^{-1}),\]
   where $\zeta_{24}$ is a primitive $24$-th root of unity. Note that $[\mb{Q}(\zeta_{24}+\zeta_{24}^{-1}): \mb{Q}(\sqrt{D})]=2$, and by transitivity of the discriminant, we have 
   \[\mathrm{Disc}(\mathbb{Q}(\sqrt{D})/\mathbb{Q})^2
   \mid \mathrm{Disc}(\mb{Q}(\zeta_{24}+\zeta_{24}^{-1})/\mathbb{Q}). \]

Finally, one can check that the discriminant of $\mb{Q}(\zeta_{24}+\zeta_{24}^{-1})$ is $2^8\cdot 3^2$. Since the discriminant of $\mathbb{Q}(\sqrt{D})$ is $D$ when $D\equiv 1\, (\mathrm{mod}\,4)$ or $4D$ when $D\equiv 2,3\, (\mathrm{mod}\,4)$, we conclude that~$D\leq6$, which is a contradiction. Therefore, the adjoint trace field of~$\mb{H}$ cannot be such a $\mb{Q}(\sqrt{D})$, as claimed. 
\end{proof}

\subsection{Differential Galois groups}
This subsection reviews some results related to the differential Galois groups of hypergeometric connections $\mathcal{H}(a;b)$, which are the Zariski closures of the monodromy groups of the corresponding local system $\mathbb{H}(a;b)$.

We call a hypergeometric connection $\mathcal{H}(a;b)$ \emph{Kummer induced} (resp. \emph{Belyi induced}, \emph{inverse Belyi induced}) if it is the direct image of a hypergeometric connection along a Kummer covering (resp. a Belyi covering of type $(a,b)$, or an inverse Belyi covering of type $(a,b)$, for a pair of positive integers $(a,b)$) of $\mathbb{G}_{m}\backslash\{1\}$ as in \cite[(3.5)]{Katz1990}). We will sometimes simply say \emph{induced} to include all these possibilities. 

\begin{thm}[{\cite[Thm.\,6.5]{beukersheckman}, also \cite[Thm.\,3.5.8]{Katz1990}}]\label{thm:beukersheckman}
    Assume that the hypergeometric connection $\mathcal{H}=\mathcal{H}(a_1, \dots, a_n ;b_1, \dots, b_{n-1}, 1)$ is irreducible (i.e. \hbox{$a_i-b_j\notin \mb{Z}$} for all $i,j$) and is neither Kummer induced nor Belyi induced nor inverse Belyi induced. Let $G$ be the differential Galois group of $\mathcal{H}$. Write $\gamma=\sum b_j -\sum a_i$. 
    \begin{enumerate}
        \item The group $G$ is reductive, 
       and  $G^\circ=G^{\circ,\mathrm{der}}$.
        \item The group $G^{\circ,\mathrm{der}}$ is either the trivial group, $\mathrm{SL}_n$, $\mathrm{SO}_{n}$, or (if $n$ is even) $\mathrm{Sp}_n$.
        \item If $\exp(2\pi \sqrt{-1} \gamma)\neq \pm 1$, then $G^{\circ,\mathrm{der}}=\{1\}$ or $\mathrm{SL}_n$.
        \item  If $\exp(2\pi \sqrt{-1} \gamma)=-1$, then $G^{\circ,\mathrm{der}}=\{1\}$ or $\mathrm{SL}_n$ or $\mathrm{SO}_n$.
        \item If $\exp(2\pi \sqrt{-1} \gamma)=1$, then $G^{\circ,\mathrm{der}}=\mathrm{SL}_n$ or (for $n$ even) $\mathrm{Sp}_n$.
    \end{enumerate}
\end{thm}

\section{Tannakian reformulation of Siegel's problem}\label{sec:reformulation}

In this section, we introduce the Tannakian categories $\mbf{H} \subset \mbf{G}$ and argue that the existence of an object of $\mbf{G}$ that does not belong to $\mbf{H}$ implies a negative answer to Siegel's problem for $G$-functions (see \cref{cor:enoughstrict}). 

\subsection{The category $\mbf{G}$}
For a Zariski open $U\subset \mb{P}^1$, let $\Conn_{\mathrm{rs}}(U)$ denote the category of integrable connections on $U$ with regular singularities along $\mb{P}^1\setminus U$. Let $\Qzmodrs$ be the category obtained as the inductive limit of $\Conn_{\mathrm{rs}}(U)$, as~$U$ ranges over all Zariski open subsets of $U \subset \mb{P}^1$ which are defined over $\Qbar$, under the natural restriction functors.

\begin{defn}[{\cite[II, \S2.3]{andre-g-functions}}]\label{defn:G}
      We define $\mbf{G}$ as the smallest full subcategory of the abelian category $\Qzmodrs$ containing all direct summands of~$\mathrm{H}^i_{\mathrm{dR}}(X/S)$ for  $X$ a smooth proper scheme over $S=\Spec(\Qbar(z))$, as well as iterated extensions of such. 
\end{defn}

As mentioned in the introduction, $\mathbf{G}$ is a subcategory of the category of $G$\nobreakdash-opera\-tors discussed with the same notation in \cite[Def.\,2.9]{fresanjossen}. In particular, objects of $\mathbf{G}$ have a basis of solutions in the differential algebra $\mathcal{G}$ from \cref{def:functions-with-monodromy}. These two categories are conjecturally the same; see \cite[Intro.]{andre-g-functions}.

\begin{prop}
    The categories $\Qzmodrs$ and $\mbf{G}$ are Tannakian, and the Tannakian group of each object $M$ of $\Qzmodrs$ is isomorphic to the differential Galois group of $M$.
\end{prop}
\begin{proof}
    The category $\mbf{G}$ is a tensor subcategory of $\Qzmodrs$, so it suffices to show that the latter is Tannakian. To do so, we need to construct a $\Qbar$-valued fiber functor $\omega\colon \Qzmodrs \to \mathrm{Vect}_{\Qbar}$. 
    
    Recall that $\Qzmodrs$ is the inductive limit of the categories $\mathrm{Conn}_{\mathrm{rs}}(U)$ of integrable connections with regular singularities on Zariski open subsets $U$ of $\mathbb{P}^1$ defined over $\Qbar$. On each $\mathrm{Conn}_{\mathrm{rs}}(U)$, the nearby cycle functor at $0$, denoted by $\Psi_0$, is a faithfully exact tensor functor, \ie a $\Qbar$-valued fiber functor. Then, the nearby cycle functor $\Psi_0$ induces a $\Qbar$-valued fiber functor on $\Qzmodrs$. 

    For the final claim, as explained in \cite[\S1.1]{KatzDiff}, all fiber functors over $\Qbar$ are (non-canonically) isomorphic, and the corresponding Tannakian groups are the differential Galois groups.
\end{proof}

\begin{rmk}
    Any object $M$ in $\mathbf{G}$ comes from some connection $M_U$ on an open $U$, \ie $M_U$ is sent to $M$ under the natural functor from $\mathrm{Conn}_{\mathrm{rs}}(U)$ to $\Qzmodrs$. By the differential Galois group of $M$, we mean that of $M_U$. This group does not depend on the choice of $U$.
\end{rmk}

Following \cite[\S\,3.6]{Andre-Gevrey-I}, a $\mathscr{D}_{\mathbb{A}^1}$-module $\mathcal{M}$, \ie a $\Qbar[z]\langle\partial_z\rangle$-module, is said to be of type $G$ if its generic fiber $\mathcal{M}[z^{-1}]$, considered as a $\Qbar(z)\langle\partial_z\rangle$-module, lies in~$\mathbf{G}$, and similarly for $\mathscr{D}_{\mathbb{P}^1}$-modules. Conversely, every object in $\mathbf{G}$ can be seen as the generic fiber of some $\mathscr{D}_{\mathbb{A}^1}$-module of type $G$. In the rest of this paper, we will not make the distinction between $\mathscr{D}$-modules of type $G$ and objects of $\mathbf{G}$.

\subsection{The category $\mbf{H}$}

\begin{defn}\label{defn:hypercat}\label{defn:H}
    We denote by $\mbf{H}$ the Tannakian subcategory of~$\Qzmodrs$ generated by 
    \begin{itemize}
        \item objects of the form $\pi_{2*}\pi_1^*\mc{H}$, where $\pi_{1}$ and $\pi_{2}$ are finite (ramified) Galois covers fitting into a correspondence 
    \[ 
        \begin{tikzcd}[column sep=small]
            & C \arrow[dl,"\pi_1"'] \arrow[dr,"\pi_2"] & \\
            \mb{P}^1_{\Qbar} & & \mb{P}^1_{\Qbar},
        \end{tikzcd}
    \]
       where $C$ is a smooth projective curve, and $\mc{H}=j_*\mathcal{H}(a;b)$ is the direct image of a hypergeometric connection by the inclusion $j\colon \mathbb{P}^1\backslash\{0,1,\infty\} \hookrightarrow \mathbb{P}^1$;
        \item connections with finite monodromies on $\mathbb{P}^1$ minus finitely many $\Qbar$-points.
    \end{itemize}
\end{defn}

\begin{prop} The category $\mbf{H}$ is contained in $\mbf{G}$. 
\end{prop} 

\begin{proof} It is immediate that a connection with finite monodromy is contained in $\mbf{G}$. On the other hand, by a theorem of Katz \cite[Thm.\,5.4.4]{Katz1990}, every hypergeometric connection~$\mathcal{H}$ can be realized in the cohomology of a family of algebraic varieties on 
$\mathbb{P}^1 \setminus \{0, 1, \infty\}$, and hence belongs to $\mathbf{G}$. If $\mathcal{H}$ is realized inside the cohomology of the generic fiber of $X \to \mathbb{P}^1$, then $\pi_1^\ast \mathcal{H}$ is realized inside the cohomology of the generic fiber of a resolution of singularities of the product $X \times_{\mathbb{P}^1} C$. Finally, it is proved in \cite[II\,\S2.3, Prop.\,2]{andre-g-functions} that direct images preserve differential modules of geometric origin, so $\pi_{2\ast}\pi_1^\ast \mathcal{H}$ belongs to $\mbf{G}$.  
\end{proof} 

In the following, we often omit the subscript $\Qbar$ when it is clear that we are considering maps and objects defined over $\Qbar$.

\begin{lem}\label{rmk:reduced}
   In \cref{defn:H},  we can assume that $\mathcal{H}$ is not of the form~$\varphi_*\mathcal{M}$ for a ramified cover $\varphi \colon X\to \mathbb{P}^1$ and a $\mathscr{D}_X$-module $\mathcal{M}$. 
\end{lem}

\begin{proof}
    Assume that $\mathcal{H}$ is of the form $\varphi_*\mathcal{M}$ for a ramified cover $\varphi \colon X \to \mathbb{P}^1$. Arguing as at the beginning of the proof of \cite[Prop.\,3.5.2]{Katz1990}, we first see that $\varphi$ is at most ramified at $0,1$, $\infty$. The statement of \cite[Prop.\,3.5.2]{Katz1990} then gives that $\mathcal{M}$ is also a hypergeometric $\mathscr{D}$-module, $X=\mathbb{P}^1$, and $\varphi$ is either a Kummer cover of~$\mathbb{G}_m$ or a Belyi or inverse Belyi map. Let $C_1, \ldots, C_k$ be the irreducible components of the normalization of the fiber product $C'=C\times_{\pi_2,\mathbb{P}^1,\varphi} \mathbb{P}^1$, and consider the compositions $\varphi_i\colon C_i\hookrightarrow C\xrightarrow{\pi_1}\mathbb{P}^1$ and $\psi_i\colon C_i \hookrightarrow C'\xrightarrow{\pi_i} \mathbb{P}^1$ as in the following diagram:
        \[ 
        \begin{tikzcd}[column sep=small]
           C_i\arrow[d,"\varphi_i"] \arrow[r, hook] \arrow[rr, bend left,"\psi_i"] & C' \ar[d]  \ar[r] & \mathbb{P}^1\ar[d,"\varphi"] \\
            \mb{P}^1 & C\arrow[l,"\pi_1"'] \arrow[r,"\pi_2"]& \mb{P}^1.
        \end{tikzcd}
    \]
    By base change, we conclude that 
    \[
        \pi_{1*}\pi_2^*\mathcal{H}=\bigoplus_{i=1}^k \varphi_{i*}\psi_i^*\mathcal{M}.
    \]
    Therefore, the differential module  
    $\pi_{1*}\pi_2^*\mathcal{H}$ is generated by those constructed from~$\mathcal{M}$ using correspondences. 
\end{proof}

\begin{lem}\label{lem:irreducibleobjects}
    Every semisimple object of $\mbf{H}$ is contained in the Tannakian subcategory generated by connections with finite monodromy and objects $\pi_{2*}\pi_1^*\mc{H}$, where~$\mc{H}$ is an irreducible hypergeometric connection which is not induced. 
\end{lem}

\begin{proof} 
    The proof is similar to that of \cite[Prop.\,3.10]{fresanjossen}. It suffices to show that any simple object $M$ of $\mathbf{H}$ is contained in the Tannakian subcategory generated by connections with finite monodromy and objects $\pi_{2*}\pi_1^*\mc{H}$, where~$\mc{H}$ is an irreducible hypergeometric connection.

    By definition of $\mathbf{H}$, the object $M$ is a subquotient of a tensor construction from hypergeometric connections and connections with finite monodromies, \ie 
        $$N\otimes \big[\pi_{2,1*}\pi_{1,1}^*\mc{H}(a_1;b_1)\big]^{\otimes m_1}\otimes \dots \big[\pi_{2,k*}\pi_{1,k}^*\mc{H}(a_k;b_k)\big]^{\otimes m_k}$$
    for some connection $N$ with finite monodromy, some integers $m_i$, some morphisms~$\pi_{2,i}$ and $\pi_{1,i}$ fitting into a  correspondence, and some hypergeometric connections $\mc{H}(a_i;b_i)$ (now $a_i$ et $b_i$ denote vectors in $\mb{Q}^{n_i}$ for some $n_i\geq 1$).  For ease of notation, we assume that $M$ is a subquotient of an object $\pi_{2*}\pi_1^*\mc{H}(a;b)$, leaving the general argument to the reader. By \cite[Cor. 3.7.5.2]{Katz1990}, the semi-simplification of~$\mc{H}(a;b)$ is given by 
    \[
        \mc{H}(a;b)^{\mathrm{ss}}\simeq \mc{H}(a_0;b_0)\oplus \bigoplus_{i=1}^r K(c_i)
    \] for some irreducible hypergeometric connection $\mc{H}(a_0;b_0)$ and some Kummer connections $K(c_i)=(\mathcal{O}_{\mathbb{G}_m},\mathrm{d}+c_i\mathrm{d}z/z)$ with rational $c_i\in \mb{Q}$.
    
    Recall that our objects are connections restricted to generic points. For the finite morphisms $\pi_i\colon C\to \mathbb{P}^1$ for $i=1,2$, the functor $\pi_1^*$ is given by base change from~$\Qbar(t)$ to $K(C)$, and the effect of the functor $\pi_{2*}$ is viewing a $K(C)\langle \partial\rangle$-module as a $\Qbar(t)\langle \partial_t\rangle$-module. So, these two functors commute with semi-simplifications and 
    \[
        (\pi_{2*}\pi_1^*\mc{H}(a;b))^{\mathrm{ss}}\simeq \pi_{2*}\pi_1^*\mc{H}(a_0;b_0)\oplus \bigoplus_{i=1}^r \pi_{2*}\pi_1^*K(c_i).
    \]
    Hence, $M$ is either a subquotient of some $\pi_{2*}\pi_1^*K(c_i)$, which are connections with finite monodromies, or a subquotient of $\pi_{2*}\pi_1^*\mc{H}(a_0;b_0)$, where now $\mc{H}(a_0, b_0)$ is irreducible. Similar to \cref{rmk:reduced}, we can assume that $\mc{H}(a_0;b_0)$ is not induced, and this completes the proof. 
\end{proof}

\begin{prop}\label{prop:solutions-H}
Every $G$-function appearing in the formulation of \cref{siegelproblem}, \ie any polynomial expression in 
\[
\mu(z)\cdot \,_{p+1}F_{p}[a_1, \dots , a_{p+1}; b_1, \dots , b_p\,|\,\lambda(z)]
\]
with coefficients in $\Qbar$, is the solution of an object in $\mbf{H}.$ 
\end{prop}

\begin{proof}
Using the fact that if $f_1, f_2 \in \mathcal{G}$ are solutions of differential modules $\mc{E}_1, \mc{E}_2$, then~$f_1+f_2$ and $f_1f_2$ are solutions of $\mc{E}_1\oplus \mc{E}_2$ and $\mc{E}_1\otimes \mc{E}_2$, it suffices to show that 
\begin{equation}\label{eq:basiccase}
    F(z)=\,_{p+1}F_{p}[a_1, \dots , a_{p+1}; b_1, \dots , b_p\,|\,\lambda(z)] 
\end{equation} is a solution of an object in $\mathbf{H}$, where $a_i \in \mathbb{Q}$, $b_j \in \mathbb{Q} \setminus \mathbb{Z}$, and $\lambda(z)$ is algebraic over~$\Qbar(z)$ and holomorphic at $0$.

Since $\lambda$ is algebraic, there exists a monic irreducible polynomial $P(t,z) \in \Qbar(z)[t]$ of minimal degree satisfying $P(\lambda(z),z)=0$. Let $X \subset \mathbb{A}^2_{t,z}$ be the zero locus of~$P(t,z)$, and denote by $p_1$ and $p_2$ the projections from $X$ to $\mathbb{A}^1_t$ and $\mathbb{A}^1_z$, respectively.

Let $C$ be the normalization of the Zariski closure of $X$ in $\mathbb{P}^2_{[T:Z:S]}$. The morphisms~$p_1$ and $p_2$ also give rise to $\pi_1$ and $\pi_2$ as follows:
\[ 
    \begin{tikzcd}[column sep=small]
        & C \arrow[dl, swap, "\pi_1"] \arrow[dr, "\pi_2"] & \\
        \mb{P}^1_{T} & & \mb{P}^1_{Z}.
    \end{tikzcd}
\]

\begin{claim}\label{claim:open-nhd}
There exist an analytic open neighborhood $U$ of $0\in \mb{P}_Z^{1,an}$ and an analytic open subset $W\subset C^{an}$ such that 
\begin{enumerate}
    \item $\pi_2|_W$ maps $W$ biholomorphically to $U$, and 
    \item $\pi_1(W)$ contains $\lambda(0)\in \mb{P}_{T}^{1,an}$.
\end{enumerate}
\end{claim}
\begin{proof}[Proof of \upshape\cref{claim:open-nhd}]
    Since $\lambda(z)$ is holomorphic at $0$, there exists a preimage $q\in C^{an}$ of $(\lambda(0), 0) \in X^{an}$, an analytic open neighborhood $W$ of $q$, and a local coordinate $z$ at $q$, such that the local parametrization on $W$ is given by $(\lambda(z), z)$ for $z$ in some open neighborhood $U$ of $0$. 
\end{proof}

Let $\nabla$ be a hypergeometric connection annihilating $\,_{p+1}F_{p}[a, b ; z]$. By the above claim, the function $F(z)$ is a solution of $(\pi_2|_W)_*(\pi_1|_W)^*\nabla|_{\pi_1(W)}$.

At last, there exists a Galois cover $\tilde\pi\colon \tilde C\to C$ such that $\pi_1\circ \tilde\pi$ and $\pi_2\circ \tilde\pi$  are both finite Galois covers. Since $(\pi_2)_*(\pi_1)^*\nabla$ is a subconnection of $(\pi_2\circ \tilde\pi)_*(\pi_1\circ \tilde\pi)^*\nabla$, the function \eqref{eq:basiccase} is also locally a solution of $(\pi_2\circ \tilde\pi)_*(\pi_1\circ \tilde\pi)^*\nabla$.
\end{proof}

Similar to \cite[Lem.\,5.6]{fresanjossen}, we conclude from \cref{prop:solutions-H} the following:

\begin{cor}\label{cor:enoughstrict}
    Suppose $M$ is an irreducible object of $\mbf{G}$ which is not in $\mbf{H}$, and moreover that $M$ is smooth at $0$. Then there is a $G$-function which is a solution of~$M$ and which cannot be written in the form specified by \cref{siegelproblem}.
\end{cor}

\begin{proof}
    Recall that $M$ admits a basis of solutions in the differential algebra $\mathcal{G}$, and denote by $\mathcal{H}yp \subset \mc{G}$ the subalgebra consisting of the functions in the formulation of \cref{siegelproblem}. For the sake of contradiction, suppose that $s(m)$ lies in $\mathcal{H}yp$  for some element~$m\in M$ and some solution $s\colon M \to \mathcal{G}$. Then \cref{prop:solutions-H} implies that there exist an object $M_0$ of $\mathbf{H}$, an element $m_0 \in M_0$, and a solution $s_0 \colon M_0 \to \mathcal{H}yp$ with $s(m)=s_0(m_0)$. We may assume that $M_0$ is generated by $m_0$.

Since $M$ is irreducible, the solution $s$ is injective. Thus, $m$ and $s(m)$ have the same annihilator ideal in $\Qbar[z]\langle\partial_z\rangle$, and this ideal contains the annihilator ideal of~$m_0\in M_0$. There is therefore a unique morphism of $\Qbar[z]\langle\partial_z\rangle$-modules $M_0\to M$ sending $m_0$ to $m$. As $m$ is non-zero and $M$ is irreducible, this morphism is surjective. Being a quotient of $M_0$, the object $M$ lies in $\mathbf{H}$, which leads to a contradiction.

Finally, $s(m)$ is a $G$-function, since it belongs to $\mathcal{G}$ and is a solution of the module~$M$ around a smooth point. 
\end{proof}

\section{Proofs of main results}\label{sec:proof}

The goal of this section is to prove \cref{thm:gfunction}, and to do so it suffices, by \cref{cor:enoughstrict}, to construct an object of $\mbf{G}$ which is not in $\mbf{H}$. To this end, we first derive a key obstruction for a rank-two object of $\mathbf{G}$ to belong to $\mathbf{H}$ in Lemma \ref{lem:fuchsiancriterion}. 

\subsection{A consequence of Goursat's lemma}

Let $\mbf{T}$ be a Tannakian category along with some choice of fiber functor $\omega$. For an object $A$ of $\mbf{T}$,  we denote by $G_A$ the Tannakian group of $A$, so that the functor $\omega$ induces a tensor equivalence 
\[\langle A\rangle^{\otimes} \cong \Rep(G_A).\]
The Lie algebra of the algebraic group $G_A$ will be denoted by $\mf{g}_A$; we may also view~$\mf{g}_A$ as the adjoint representation of $G_A$, and going through the above equivalence we obtain an object of $\langle A\rangle ^{\otimes}$, denoted by $\underline{\mf{g}}_A.$

\begin{defn}
For objects $A$ and $B$ of a Tannakian category $\mbf{T}$, let $\mf{g}_B$ and~$\mf{g}_{A\oplus B}$    denote the Lie algebras of the Galois groups of $B$ and $A\oplus B$, respectively. We say that~$A$ is \emph{Lie-generated} by $B$ if the natural morphism of Lie algebras 
\[
\mf{g}_{A\oplus B}\longrightarrow \mf{g}_{B}
\] is an isomorphism. If $A$ is Lie-generated by $B$, then there is a surjection $\mathfrak{g}_{B}\twoheadrightarrow \mathfrak{g}_{A}$.
\end{defn}

We will make use of the following version of Goursat's lemma for Lie algebras:

\begin{lem}[{\cite[Lem.~4.5]{fresanjossen}}]\label{lemma:singlegeneration}
    Let $\mbf{T}$ be a Tannakian category generated by a family of objects $\mc{B}$. Let $A$ be an object of $\mbf{T}$ such that $\mf{g}_A$  is simple and non\nobreakdash-commutative. Then $A$ is Lie-generated by a single object of the family $\mc{B}$. 
\end{lem}

We now come to the main criterion of our work.
\begin{lem}[Key criterion]\label{lem:fuchsiancriterion}
    Let  $A$ be a rank-two object of $\mbf{G}$  with differential Galois group $\mathrm{SL}_2$. Suppose that $A$ is Lie-generated by an object of the form  $\pi_{2*}\pi_1^*\mc{H}$, where $\mc{H}$ is an irreducible hypergeometric connection, which is not induced, and $\pi_{1}$ and $\pi_{2}$ are maps fitting into a diagram as in \cref{defn:hypercat}. Then $\mc{H}$ is of rank two with $\mf{g}_{\mc{H}}\cong \mf{sl}_2$. 
    
    Furthermore, the adjoint trace field of $A$, \ie the trace field of $\ad^0(A)$, 
    is isomorphic to the adjoint trace field of $\mc{H}$. Here $\ad^0(A)$ denotes the connection given by the traceless endomorphisms of $A$.
\end{lem}

\begin{proof}
    By assumption, $A$ is Lie-generated by $\pi_{2*}\pi_1^*\mc{H}$, and hence $\pi_2^*A$ is Lie\nobreakdash-generated by $\pi_2^*(\pi_{2*}\pi_1^*\mc{H})$, as objects of $\mbf{G}$. Since $\pi_2$ is a finite Galois cover,  by Mackey's induction-restriction formula we have 
    \[
    \pi_2^*(\pi_{2*}\pi_1^*\mc{H}) = \bigoplus_{\sigma} (\pi_1^*\mc{H})^{\sigma},
    \]
    where the sum is over automorphisms $\sigma\colon C\rightarrow C$  covering $\pi_2$, and the superscript $^{\sigma}$ denotes the pullback under $\sigma$. 
    
    For brevity, let us write $\mc{V}\defeq \pi_1^*\mc{H}$. Then $\mf{g}_{\mc{V}^{\sigma}} \cong \mf{g}_{\mc{V}} \cong \mf{g}_{\mc{H}}$, where the last isomorphism follows from \cite[Prop.\,1.4.4]{KatzDiff}. On the other hand,  $\mf{g}_A\cong \mf{sl}_2$ is a quotient of $\bigoplus_{\sigma} \mf{g}_{\mc{V}^{\sigma}}$. Since $\mathcal{H}$ is irreducible, $\mf{g}_{\mc{H}}$ is a semisimple Lie algebra by \cref{thm:beukersheckman}; as $\mathfrak{g}_{\mathcal{H}}$ surjects to $\mf{sl}_2$, by inspection of the groups appearing in \cref{thm:beukersheckman} we deduce that $\mathcal{H}$ is of rank two with $\mf{g}_{\mc{H}}\cong\mf{sl}_2$.

    It remains to prove the statement about trace fields, and for this, we may pass to a non-empty open subset  $U \subset \mathbb{P}^1$ such that $A$ is smooth on $U$, and $\bigoplus_{\sigma} \mc{V}^{\sigma}$ is smooth on $\pi_2^{-1}(U) \subset C$. Let $F_{\mathrm{ad}^0(\mc{H})}$ be the adjoint trace field of $\mc{H}$, and similarly for  $\pi_1^*\mathcal{H}$. So for any $\sigma$, the trace field of $\underline{\mf{g}}_{\mc{V}^{\sigma}}$ is the same as $F_{\mathrm{ad}^0(\pi_1^*\mathcal{H})}$, which is a priori a subfield of $F_{\mathrm{ad}^0(\mc{H})}$. On the other hand, there is a surjective morphism 
    \[
        f\colon   \bigoplus_{\sigma} \underline{\mf{g}}_{\mc{V}^{\sigma}} \longrightarrow \underline{\mf{g}}_{\pi_2^*A},
    \]
    where $\underline{\mf{g}}_{\mc{V}^{\sigma}}$ and~$\underline{\mf{g}}_{\pi_2^*A}$ are regarded as objects of the Tannakian subcategories $\langle \mc{V}^{\sigma}\rangle^\otimes$ and $\langle \pi_2^*A\rangle ^\otimes$ of $\mbf{G}$. Each $\underline{\mathfrak{g}}_{\mathcal{V}^\sigma}$ is simple and $f$ is surjective, so~$\underline{\mf{g}}_{\pi_2^*A}$ is isomorphic to a direct sum of some of the $\underline{\mf{g}}_{\mc{V}^{\sigma}}$. From this we deduce an isomorphism $\underline{\mf{g}}_{\pi_2^*A}\cong \underline{\mf{g}}_{\mc{V}^{\sigma}}$ for one automorphism $\sigma$ since both objects have the same rank. Therefore, the  trace field of $\underline{\mf{g}}_{\pi_2^*A}$ is isomorphic to that of  $F_{\mathrm{ad}^0(\pi_1^*\mathcal{H})}$.
    
    Finally, \cref{lemma:comminvariance} implies that the adjoint trace fields of $A$ and $\mc{H}$ are both invariant under passing to finite index subgroups, as we now explain. It suffices to show this for the Galois conjugates $\sigma(A)$ and $\sigma(\mc{H})$ for some choice of $\sigma\in \Aut(\mb{C})$; here, by Galois conjugates we mean the connections corresponding to $\sigma$ applied to the local systems for $A$ and $\mc{H}$, respectively. Since $A$ has infinite monodromy, it has a Galois conjugate which is not unitary, and for ease of notation we assume this is $A$ itself; then $\mc{H}$ must also be non-unitary, as it Lie-generates $A$. Now both $A$ and $\mc{H}$ satisfy the assumptions of  \cref{lemma:comminvariance}, and so their adjoint trace fields are commensurability invariants.
    
    In conclusion,  we now have
    \[
    F_{\mathrm{ad}^0(A)} = F_{\underline{\mf{g}}_A} = F_{\underline{\mf{g}}_{\pi_2^*A}} \simeq F_{\mathrm{ad}^0(\pi_1^*\mathcal{H})}=  F_{\mathrm{ad}^0(\mc{H})}, 
    \]
    as required.
\end{proof}

It remains to construct explicitly a rank-two object $A$ of $\mbf{G}$ as in \cref{lem:fuchsiancriterion} that violates its conclusion, \ie whose adjoint trace field cannot be isomorphic to that of a rank-two hypergeometric connection.

\subsection{Recollections on  Shimura curves}

We first collect some preliminaries regarding Shimura curves. We will try to take a down-to-earth approach here as that is all that is necessary for our applications. For more details on quaternion algebras, Shimura curves, and other things,   we refer the reader to the excellent and comprehensive book of Voight's \cite[\S 38]{voight-big-book}.

Let $F$ be a totally real number field and $B$ a quaternion algebra over $F$ which is split at a unique infinite place $\sigma\colon F\xhookrightarrow{} \mb{R}$. There is, therefore,  an isomorphism
\[
B\otimes_{F, \sigma} \mb{R} \cong M_2(\mb{R}). 
\]
There is a unique involution $\bar{~}\colon B\rightarrow B$, known as the standard involution (see \cite[\S 3.2]{voight-big-book}), so that for any $\gamma \in B$, the \emph{reduced norm} $\nrd{\gamma}:= \gamma\bar{\gamma}$ lies in $F$. The extension of scalars of $\nrd$ to $B\otimes_{F, \sigma}\mb{R}= M_2(\mb{R})$ is simply the determinant. For any order $\mc{O}\subset B$, let $\mc{O}_1\subset \mc{O}$ denote the elements of reduced norm 1; there are maps  
\[
\mc{O}_1^{\times} \longrightarrow (B\otimes_{F, \sigma} \mb{R})^{\times}_{\nrd =1}=\SL_2(\mb{R}) \longrightarrow \PSL_2(\mb{R}).
\]
Let $\Gamma\subset \PSL_2(\mb{R})$ denote the image of this composition, so that the above map factors as $\mc{O}_1^{\times} \rightarrow \mc{O}_1^{\times}/\{\pm 1\}\simeq \Gamma\subset \PSL_2(\mb{R})$.  

\begin{prop}
    The quotient $\Gamma\backslash \mf{H}$ is in bijection with the complex points of a complex orbicurve $\Sh_{B, \mc{O}}$; moreover, the latter can be defined over $\Qbar$. 
\end{prop}
   \begin{proof}
       For a nice exposition, we refer the reader to \cite[p.\,36]{milne-curves}.
   \end{proof}
   
\begin{defn}
    We refer to $\Sh_{B, \mc{O}}$ as the Shimura curve attached to $(B/F, \mc{O})$. 
\end{defn}

\begin{prop}\label{prop:sh unif loc sys}
    Let $S\subset \Sh_{B, \mc{O}}$ be the set of orbifold points, and $\Sh^{\circ}_{B, \mc{O}}$ the complement of $S$. Consider the representation 
    \[
    \rho\colon  \pi_1(\Sh^{\circ}_{B, \mc{O}})\rightarrow \PSL_2(\mb{R})\rightarrow \GL_3(\mb{C}),
    \] with the last map being the adjoint representation of $\PSL_2(\mb{R})$. Then
    \begin{itemize}
        \item  $\rho$ has trace field equal to $\sigma(F)$;
        \item if $\Shbo^{\circ}$ is rational, then $\rho$ may be lifted to  a representation 
        \begin{equation}\label{eq:geometriclifts}
        \tilde{\rho}\colon \pi_1(\Shbo^{\circ})\longrightarrow \SL_2(\mb{R}). 
        \end{equation} Moreover, any such lift is of geometric origin.
    \end{itemize}
\end{prop}

\begin{proof}
    The first part follows from   \cite[Thm.\,8.3.9]{maclachlan2003arithmetic} and Lemma 3.5.6 of \textit{loc.\,cit.}.

    For the second part, if  $\Shbo^{\circ}$ is rational, then its fundamental group is free, so one lifts $\rho$ by choosing any lift to $\SL_2(\mb{R})$ of the images by $\rho$ of the generators of~$\pi_1(\Sh^{\circ}_{B, \mc{O}})$. For the claim that such a $\tilde{\rho}$ is of geometric origin, see for example~\cite[Thm.\,9.3]{corlettesimpson} (where the authors treat the more general case of quaternionic Shimura varieties), combined with the fact that any polarized integral variation of Hodge structures of weight one comes from a family of abelian varieties.
\end{proof}

\begin{defn}\label{defn:unif loc sys}
    In the case when $\Shbo^{\circ}$ is rational, we refer to any of the rank-two local systems \eqref{eq:geometriclifts} as a \emph{uniformizing local system}.  
\end{defn}

\subsection{Non-hypergeometric $G$-functions from Shimura curves}

We are now able to prove the main theorem from the introduction. 

\begin{proof}[Proof of \upshape\cref{thm:gfunction}]

By combining \cref{lemma:singlegeneration} and \cref{lem:fuchsiancriterion} with \cref{prop:trace-field-hyp}, to prove \cref{thm:gfunction}, it suffices to find a non-unitary local system of geometric origin with differential Galois group $\SL_2$ and non-abelian adjoint trace field on a punctured projective line. 
    
\begin{claim}\label{claim:shimura}    Suppose $\Shbo$ is a Shimura curve which 
    \begin{itemize}
    \item is rational, 
        \item and is associated with a non-split quaternion algebra $B$ over a totally real cubic number field $F$ which is non-abelian (equivalently, the Galois closure of $F$ has Galois group $S_3$ over $\mb{Q}$).
    \end{itemize}
    Let $\mb{L}$ be a uniformizing  rank-two  local system on $\Shbo^{\circ}$, as in \cref{defn:unif loc sys}. Then   the adjoint trace field of $\mb{L}$ is $\sigma(F)\subset \mb{C}$ for some embedding $\sigma\colon F\hookrightarrow \mb{C}.$ In particular, the adjoint trace field of $\mb{L}$ is non-abelian. Moreover,  $\mathbb{L}$ has $\SL_2$-monodromy.
    \end{claim}
  
    \begin{proof}[Proof of~\cref{claim:shimura}] 
    The claim about trace fields is contained in \cref{prop:sh unif loc sys}; for the second part, by construction $\mb{L}$ is Zariski dense in $\SL_2(\mb{R})$, as required. 
    \end{proof}
    
    For explicit examples of Shimura curves satisfying the assumptions of \cref{claim:shimura}, we refer the reader to  \cite[Table 4.3]{voight}; for example, we can take the first entry of the second column, where the cubic field $F$ has discriminant $148$, and the signature of the Shimura curve is $(0; 2^3, 3)$. The latter notation means that the Shimura curve is rational with four orbifold points with stabilizers $\mb{Z}/2\mb{Z}, \mb{Z}/2\mb{Z}, \mb{Z}/2\mb{Z}, \mb{Z}/3\mb{Z}$, respectively. Note that $F$ is a non-abelian extension since the discriminant $148$ is not a square: indeed, the Galois closure of a cubic number field has Galois group either $\mb{Z}/3\mb{Z}=A_3$ or $S_3$ and both cases are distinguished by the discriminant being a square or not. This concludes the proof of \cref{thm:gfunction}.
\end{proof}

\begin{rmk}
    In fact, there is a plentiful supply of counterexamples coming from such Shimura curves. For example, just for Shimura curves attached to cubic fields~$F$ as above, \cite[Table 4.3]{voight} gives 45 examples; amongst these examples, the biggest discriminant of $F$ is 1593.
\end{rmk}

\begin{rmk}
    As mentioned in the introduction, any uniformizing local system on a Shimura curve satisfying the hypotheses of \cref{claim:shimura} gives a counterexample to Dwork's conjecture \cite[p.\,784]{dwork1990differential}. In this way, we provide many new counterexamples to this conjecture.
\end{rmk}

Using \cref{lem:fuchsiancriterion}, we may rule out even more Shimura curve examples by arguing as Krammer does in \cite[\S11]{krammer}:
\begin{prop}\label{prop:Krammer-example}
    Suppose $\Shbo$ is a rational Shimura curve attached to a quaternion algebra $B$ over $\mb{Q}$, whose discriminant is not $(1)$ or $(2)(3)$. Then any of the uniformizing rank-two local systems on $\Shbo^{\circ}$ is not  an object of $\mbf{H}$.
\end{prop}

Note that we slightly abuse terminology and say that \enquote{$\mb{L}$ is not an object of $\mbf{H}$} to mean the same for  the associated connection under Riemann--Hilbert. 

\begin{proof}    
    Suppose $\mb{L}$ is a uniformizing rank-two local system on $\Shbo^{\circ}$, and that moreover, it is an object of $\mbf{H}$. Note that, by the construction of Shimura curves, the image of the associated representation in $\PSL_2(\mb{R})$ is a discrete subgroup. By \cref{lem:fuchsiancriterion}, we have that $\mb{L}$ is Lie-generated by $\pi_{2*}\pi_1^{*}\mb{H}$ for some rank-two hypergeometric local system $\mb{H}$ and covers $\pi_{1}$ and $\pi_{2}$ fitting into a correspondence as in \cref{defn:hypercat}. This implies that the monodromy group $P$ (\ie the image of the associated representation, and not the Zariski closure) of $\mb{H}$ is commensurable to that of $\mb{L}$. Since the hypergeometric monodromy groups that are Fuchsian are precisely the triangle groups, we deduce that $P$ is 
    \begin{itemize}
        \item commensurable with a triangle group, and 
    \item arithmetic.
    \end{itemize}
    The arithmetic triangle groups were classified by Takeuchi \cite{takeuchi1977commensurability}, and the only ones associated with quaternion algebras over $\mb{Q}$ are the ones with discriminants $(1)$ and $(2)(3)$, which contradicts  our assumption. Therefore, $\mathbb{L}$ is not in $\mathbf{H}$.
\end{proof}

\begin{example}
    Krammer gives an example of a Shimura curve satisfying the assumptions of \cref{prop:Krammer-example}. In this case, $B$ is the unique quaternion algebra over $\mb{Q}$ ramified at the primes $3$ and $5$, and $\mc{O} \subset B$ is a certain explicit order as given in \cite[\S10, Eq.\,(11)]{krammer}. Then $\Shbo^{\circ}=\mb{P}^1\setminus \{0,1, 81, \infty\}$. One choice of the uniformizing differential equation is given as Equation (10) of \textit{loc.\,cit.}, and centering this  at $x=-1$ we obtain \eqref{eq:counter-example} in the introduction.
\end{example}

\subsection{Examples from Teichm\"uller curves}

In this section, we give yet more examples of non-hypergeometric $G$-functions, this time coming from Teichm\"uller curves. Recall that a \emph{Teichm\"uller curve} is a smooth complex algebraic curve $X$ along with a morphism $X\rightarrow \mc{M}_g$ for some $g \geq 2$ such that $X$ is geodesic for the Teichm\"uller metric on $\mc{M}_g$. These are very special curves  in $\mc{M}_g$, and we refer the reader to~\cite[\S 1.3]{moller2006variations} for other characterizations as well as  detailed properties of such curves. Let $\pi\colon  \mscr{C}\rightarrow X$ be the pullback of the universal family of curves over~$\mc{M}_g$. The uniformization of $X$ gives a representation $\pi_1(X)\rightarrow \PSL_2(\mb{R})$, which we refer to as the uniformizing representation.

\begin{prop}[{\cite[Prop.\,2.4]{moller2006variations}}]
The local system $R^1\pi_*\mb{C}$ has  a rank-two local system $\mb{L}$, with trivial determinant, as a direct summand which is uniformizing in the following sense: if we write $\rho_{\mb{L}}\colon \pi_1(X)\rightarrow \SL_2(\mb{C})$ for the representation corresponding to $\mb{L}$, then the composition 
\begin{equation}\label{eqn: projection-uniform}
   \pi_1(X)\rightarrow \SL_2(\mb{C})\rightarrow \PSL_2(\mb{C}) 
\end{equation} is conjugate to the uniformizing representation of $X$.
\end{prop}

We refer to such an $\mb{L}$ as a tautological local system on $X$.

\begin{prop}\label{prop: teich-criterion}
    Suppose $X$ is a rational Teichm\"uller curve, and $\mb{L}$ is one of its tautological local systems. If the adjoint trace field of $\mb{L}$ is $\mb{Q}(\sqrt{D})$ for some squarefree integer $D\geq 7$, then $\mb{L}$ does  not correspond to an object of $\mbf{H}$ under the Riemann--Hilbert correspondence. 
\end{prop}
\begin{proof}
    As  the tautological local system $\mb{L}$ corresponds to the uniformizing representation (under the composition \eqref{eqn: projection-uniform}), it  in particular 
    has Zariski-dense monodromy. Let $\mc{L}$ denote the connection associated with $\mb{L}$, and assume that $\mc{L}$ is an object of~$\mbf{H}$. Then \cref{lem:fuchsiancriterion} implies that $\mc{L}$ is Lie-generated by an object of the form $\pi_{2*}\pi_1^*\mc{H}$ (with notation as in the statement of the lemma), where $\mc{H}$ is a rank-two hypergeometric connection with the same trace field as that of $\mb{L}$, namely $\mb{Q}(\sqrt{D})$ for some $D\geq 7$. On the other hand,   \cref{prop:trace-field-hyp} says that no such $\mc{H}$ exists, and hence $\mc{L}$ cannot be an object of $\mbf{H}$, as required. 
\end{proof}

\begin{rmk}
   The set of Teichm\"uller curves to which \cref{prop: teich-criterion} applies is non-empty: examples were constructed by McMullen in \cite[Thm.\,9.8]{mcmullen2003billiards}. The required claim about the trace fields of these Teichm\"uller curves follows from the proof of Cor.\,7.3 of \emph{loc.\,cit.}. We  refer the reader to \cite[\S 1]{bouwmoller} for a summary of the properties of these Teichm\"uller curves.  
\end{rmk}

\section{Infinitely many non-hypergeometric \texorpdfstring{$G$}{G}-functions of order two}\label{sec:infinite}

In this final section, we give infinitely many examples of non-hypergeometric $G$\nobreakdash-functions of differential order $2$, thereby answering a question raised by Krammer in the context of his counterexample to Dwork's conjecture; see \cite[\S12]{krammer}. We start with a very rough sketch of the argument to help orient the reader. 
 
 Let~$\mc{M}_{0, 4}$ denote the moduli space of curves of genus $0$ with $4$ marked points, and~$\pi\colon \mc{C}_{0, 4} \rightarrow \mc{M}_{0, 4}$ the universal $4$-punctured genus $0$ curve. The idea is to consider a local system $\mb{V}$ on $\mc{C}_{0,4}$ arising from a family of abelian varieties, which was constructed in a previous work of Lam and Litt, and to show that its restrictions to most fibers of $\pi\colon \mc{C}_{0,4}\rightarrow \mc{M}_{0,4}$ do not come from algebraic pullbacks of hypergeometric local systems. Indeed, if the contrary holds, then our trace field argument implies that the restrictions of~$\mb{V}$ to infinitely many fibers come from a single hypergeometric local system. This in turn implies that the image of $\mc{C}_{0,4}$ under the period mapping associated with~$\mb{V}$ intersects Hecke translates of a fixed curve at many points, \ie there are many \emph{unlikely} intersections. This violates the Andr\'e--Pink --Zannier conjecture (which is a theorem of Richard and Yafaev in our case of interest) in the theory of unlikely intersections, from which we may conclude.

\subsection{A family of rank-two local systems}

\begin{prop}\label{prop:localsystems}
    For each integer $m\geq 1$, there exists a dominant étale morphism $h \colon \widetilde{\mc{M}}\to \mc{M}_{0,4}$ fitting in a Cartesian diagram 
    \begin{equation}\label{eqn:mcg finite}
        \begin{tikzcd}
\widetilde{\mc{C}} \arrow[r] \arrow[d]
& \mc{C}_{0,4} \arrow[d] \\
\widetilde{\mc{M}}\arrow[r, "h"]
&  \mc{M}_{0,4}
\end{tikzcd}
    \end{equation}
    and a rank-two complex local system $\mb{V}$ with trivial determinant on $\widetilde{\mc{C}}$ satisfying: 
\begin{enumerate} 
\item for each closed point $p$ of $\Mtil$, the Zariski closure of the image of $\mb{V}|_{\Ctil_p}$ is~$\SL_2$, and the local monodromy is conjugate to 
    \[
    \begin{pmatrix}
        1 & 1 \\
        0 & 1
    \end{pmatrix}
    \]
    at three of the punctures of $\Ctil_p$, and to  \[
    \begin{pmatrix}
        -1 & 1 \\
        0 & -1
    \end{pmatrix}
    \]
    at the remaining puncture. Moreover, $\mb{V}|_{\Ctil_p}$ has trace field $\mb{Q}(\zeta +\zeta^{-1})$, where~$\zeta$ is a primitive $m$-th root of unity;
       \item there exists a family of principally polarized abelian varieties $\pi\colon \mc{A}\rightarrow \widetilde{\mc{C}}$ such that the local system $\mb{V}$ is a direct summand of $R^1\pi_*\mb{C}$. For~$m\geq 5$, this local system is not pulled back from a curve, \ie $\mb{V}$ is not isomorphic to $b^*\mb{U}$ for any map $b\colon \Ctil \rightarrow \mc{B}$ to a smooth curve  $\mc{B}$, and any local system $\mb{U}$ on~$\mc{B}$. 
\end{enumerate}
\end{prop}

\begin{proof}
    This follows from \cite[Thm.\,1.1.4, Thm.\,1.1.6]{lamlitt} and their proofs, as we now explain. The rank-two local systems of these results of \emph{loc.\,cit.} are obtained as follows: for each $m\geq 1$ and  $\lambda\in \mb{C}\setminus \{0,1\}$, consider $\mathrm{MC}_{-1}(\mb{W})$, where 
    \begin{itemize}
    \item $\mathrm{MC}_{-1}$ denotes Katz's middle convolution functor with respect to the unique non-trivial rank-one local system of order $2$ on $\mathbb{G}_m$. For our purposes, we simply take $\mathrm{MC}_{-1}$ to mean  a certain functor from the category of local systems on $\mb{P}^1\setminus \{0,1,\lambda, \infty\}$ to itself, though it can be phrased much more generally in the language of perverse sheaves: for a definition of this functor which stays in the realm of local systems, see \cite[Defn.\,3.1.1]{lam-landesman-litt};
        \item $\mb{W}$ is given by the direct image $f^{\circ}_*\mb{L}$ where  $f\colon E\rightarrow \mb{P}^1$ is the double cover ramified at $0,1,\lambda, \infty$, and $f^{\circ}$ denotes the restriction $f|_{{E\setminus \{f^{-1}\{0,1,\lambda, \infty\}\}}}$; 
        \item $\mb{L}$ is the restriction to $E\setminus \{f^{-1}\{0,1,\lambda, \infty\}\}$ of a torsion rank-one local system of  order $m$ on the elliptic curve $E$.
    \end{itemize}
    Then $\mathrm{MC}_{-1}(\mb{W})$ is a local system on $\mb{P}^1\setminus \{0,1,\lambda, \infty\}$.  The computations of the rank and the local monodromies of $\mathrm{MC}_{-1}(\mb{W})$ are recorded in \cite[Prop.\,3.1.1]{lamlitt}; note that it has 
     trivial determinant. Moreover, it is MCG-finite, \ie has a finite orbit under the action of the mapping class group of $\mb{P}^1\setminus \{0,1,\lambda, \infty\}$, since the functor $\mathrm{MC}_{-1}$ is equivariant for the action of the mapping class group and local systems with finite monodromy are MCG-finite. Hence, $\mathrm{MC}_{-1}(\mb{W})$ extends to a diagram of the form \eqref{eqn:mcg finite} by \cite[Cor.\,2.3.5]{landesman-litt-canonical}; note that \textit{loc.\,cit.} makes the assumption that $g\geq 1$, but this assumption is used only through \cite[Lemma.\,2.3.2]{landesman-litt-canonical}, which is automatic in our case since $\mathrm{MC}_{-1}(\mb{W})$ has trivial determinant. It is also possible to give a  direct argument for the existence of the morphism $h\colon \Mtil \rightarrow \mc{M}_{0,4}$ as in~\cite[proof of Thm.\,5.5.4]{Katz96rigid}.
    
    We now prove  the claim about trace fields,  or equivalently, by \cref{prop:field of defn} the fields of definition.  Note that, as can be seen from its definition,  $\mathrm{MC}_{-1}$ is Galois equivariant: more precisely,  
    \[
    \mathrm{MC}_{-1}(\mb{M}\otimes_{\mb{C}, \sigma} \mb{C}) \simeq \mathrm{MC}_{-1}(\mb{M})\otimes_{\mb{C}, \sigma} \mb{C},
    \]
    for any local system $\mb{M}$ on $\mb{P}^1\setminus \{0,1,\lambda, \infty\}$, and any automorphism $\sigma\colon \mb{C}\rightarrow \mb{C}$.
    
    Therefore, it remains to check that $f^{\circ}_*\mb{L}$ has  field of definition $\mb{Q}(\zeta+\zeta^{-1})$, for~$\zeta$ a primitive $m$-th root of unity. This is immediate from the definitions and the assumption that $\mb{L}$ has order $m$.

    The claim that the local system $\mb{V}$ comes from a family of abelian varieties follows from \cite[Thm.\,1.1.6]{lamlitt}.
    
    Finally, we check that  $\mb{V}$ is not pulled back from a curve. Recall that the restriction of $\mb{V}$ to each fiber $\mb{P}^1\setminus \{0,1,\lambda, \infty\}$ is a MCG-finite local system, and such $\mb{V}$ that are of ``pullback type'', \ie pulled back from a curve, may be classified following \eg the analysis in 
    \cite{diarra}, as we now explain.
    Suppose for the sake of contradiction that $\mb{V}$ is of pullback type, \ie there exists $\mc{B}, b$ and $\mb{U}$ as in the statement of the proposition. Note that $\mb{U}$ must be irreducible, and therefore rigid, since this is true of any irreducible rank-two local system on $\mb{P}^1\setminus \{0,1,\infty\}$.
    
    Then Diarra \cite[Prop.\,2.6, Prop.\,2.7]{diarra} shows that the curve $\mc{B}$ may be taken to be $\mb{P}^1\setminus \{0, 1,\infty\}$. Now let $p_0, p_1, p_{\infty}\in \mb{Z}_{\geq 1} \cup \{\infty\}$ be the orders of the monodromies of $\mb{U}$ around $0, 1, \infty$, respectively; let us assume without loss of generality that $p_0\leq p_1\leq p_{\infty}$.  
    Since, for any $p \in \Mtil$, the local system $\mb{V}|_{\Ctil_p}$ has infinite monodromies at all four punctures, we must have $p_{\infty}=\infty$.
    
    Now   
    \cite[Table following Remarque 3.3]{diarra} gives the possible values of   $p_0, p_1$, as well as  $d:= \deg(b|_{\Ctil_p})$. If our $(p_0, p_1, p_{\infty})$ is in the second, third, or fourth row, the trace field of $\mb{U}$, and hence that of $\mb{V}$, would be $\mb{Q}$, contradicting the assumption~$m\geq 5$. Therefore, $(p_0, p_1, p_{\infty})$ must be in the first row, so that $d=2$. This forces $p_1=\infty$, since $\mb{V}|_{\Ctil_p}$ has four cusps with  infinity monodromy. This implies the trace field of $\mb{U}$ is $\mb{Q}$, again a contradiction, as required.   
\end{proof}

\subsection{The Andr\'e--Pink--Zannier conjecture} We will next show that the existence of only finitely many non-hypergeometric $G$-functions of order $2$ would contradict the Andr\'e--Pink--Zannier conjecture in the cases proven by Richard and Yafaev. 

\begin{thm}[{\cite[Thm.\,1.3]{richard2021generalised}}]\label{thm:apz}
    Let $s_0$ be a point of a Shimura variety of abelian type~$\Sh_K(G, X)$. Let $Z\subset \Sh_K(G, X)$ be a subvariety whose intersection with the generalized Hecke orbit of $s_0$ is Zariski dense. Then, $Z$ is a finite union of weakly special subvarieties.
\end{thm}

Let us briefly introduce the objects in the statement, referring the reader to~\cite[\S2]{richard2021height} for more details on generalized Hecke orbits. Recall that the Shimura variety $S=\Sh_K(G, X)$ is the double quotient $G(\mb{Q}) \backslash (X \times G(\mathbb{A}_f)) \slash K$, where $(G, X)$ is a Shimura datum and~$K$ a compact open subgroup of $G(\mathbb{A}_f)$. A point of $X$ is given by a morphism of real algebraic groups $\mb{S}:= \mathrm{Res}_{\mb{C}/\mb{R}}(\mathbb{G}_m) \to G_{\mathbb{R}}$, and its Mumford--Tate group is the smallest $\mb{Q}$-algebraic subgroup of $G$ containing its image. We normalize $(G, X)$ in such a way that $G$ is the generic Mumford--Tate group of~$X$. Each representation of $G$ gives rise to a local system on $S$, and we refer to such a local system as  \emph{tautological}.

Let $s_0=[x_0, g_0]$ be a point of $S$ and let $M$ be the Mumford-Tate group of~$x_0$. The \emph{generalized Hecke orbit} of $s_0$ is the set of $[\varphi \circ x_0, g]$ with $g \in G(\mathbb{A}_f)$ and $\varphi \in \Hom(M, G)$ satisfying $\varphi \circ x_0\in X$. 

Very roughly, if $S$ parametrizes abelian varieties with extra Hodge tensors, one may think of two points being in the same generalized Hecke orbit if the corresponding abelian varieties are isogenous, with no extra constraints on the Hodge tensors. Along the same lines, we have the following:

\begin{prop}\label{prop: iso-hs-generalized-hecke}
Let $\mb{V}_{S, \mb{Q}}$ be one of the tautological $\mb{Q}$-variations of Hodge structures  on $S=\Sh_K(G, X)$. Suppose $s_0, s_1 \in S$ are such that there is an isomorphism of $\mb{Q}$-Hodge structures:
\[
\mb{V}_{S, \mb{Q}, s_0}\simeq \mb{V}_{S, \mb{Q}, s_1}.
\]
Then, $s_1$ is in the generalized Hecke orbit of $s_0$.
\end{prop}
\begin{proof}
Suppose $s_i=[x_i, g_i]$. 
For $i=0,1$, let $M(s_i)$ denote the Mumford--Tate group of $\mb{V}_{S, \mb{Q}, s_i}$. Using the morphisms $x_i$, we have maps of $\mb{Q}$-groups
\[
\phi_i: M(s_i) \rightarrow G.
\]
The isomorphism $\mb{V}_{S, \mb{Q}, s_0}\simeq \mb{V}_{S, \mb{Q}, s_1}$ induces an isomorphism $M(s_0)\xrightarrow[]{\iota} M(s_1)$ such that the diagram
\begin{equation}
    \begin{tikzcd}
  M(s_0)  \arrow[r, "\phi_0"] \arrow[d, "\iota"] & G \arrow[d, "\mathrm{id}"]\\
  M(s_1) \arrow[r, "\phi_1"] & G
\end{tikzcd}
\end{equation}
commutes. Then, taking $\varphi:= \phi_1 \circ \iota$ in the definition of the generalized Hecke~orbit suffices.
\end{proof}

\subsection{Main result}

We keep notation from \cref{prop:localsystems}. 

\begin{thm}\label{thm: infinitely-many-examples}
    Let $m$ be an integer such that $\phi(m)/2$ is odd $\geq 3$. Then there exists a finite set $S_{\exc}\subset \Mtil(\Qbar)$ such that   for any $p\in \Mtil(\Qbar)\setminus S_{\exc}$, the vector bundle with connection corresponding to the local system $\mb{V}|_{\Ctil_p}$ is not an object of~$\mbf{H}.$ 
    
    Moreover, we can choose infinitely many points $p_1, p_2, \dots \in \Mtil(\Qbar)\setminus S_{\exc}$ such that 
         for~$i\neq j$, the local systems $\mb{V}|_{\Ctil_{p_i}}$ and  $\mb{V}|_{\Ctil_{p_j}}$ cannot be identified with each other via an automorphism of $\mb{P}^1$ and tensoring by a rank-one local system.
\end{thm}

\begin{proof}
We first show that there exists a finite set $S_{\exc}$ for which the first condition holds. Suppose the contrary, so that there are infinitely many points $p_i\in \Mtil(\Qbar)$ such that the connection associated with~$\mb{V}|_{\Ctil_{p_i}}$ belongs to $\mbf{H}.$ For brevity, we write~$\mb{V}_i$ for the  local system $\mb{V}|_{\Ctil_{p_i}}$ on $\Ctil_{p_i}$. 

    By \cref{lemma:singlegeneration} and the proof of  \cref{lem:fuchsiancriterion} (in particular, the sentence \enquote{...we deduce an isomorphism $\underline{\mf{g}}_{\pi_2^*A}\cong \underline{\mf{g}}_{\mc{V}^{\sigma}}$...}),  for each $i$, there is 
    \begin{itemize}
    \item 
    a correspondence diagram 
     \[ 
        \begin{tikzcd}[column sep=small]
            & C_i \arrow[dl,"\pi_1"'] \arrow[dr,"\pi_2"] & \\
            W_i & & \Ctil_{p_i}^{\circ},
        \end{tikzcd}
    \]
    with $\pi_{1}$ and $\pi_{2}$ finite \'etale and Galois, where $W_i\subset \mb{P}^1\setminus \{0,1, \infty\}$, and $\Ctil_{p_i}^{\circ}\subset \Ctil_{p_i}$ are non-empty Zariski open subsets,
    \item a local system $\mb{H}_i$ on $W_i$ which is the restriction of a rank-two hypergeometric local system on $\mb{P}^1\setminus \{0,1,\infty\}$, such that
    \begin{equation}\label{eqn:ad iso}
    \ad^0(\pi_2^*\mb{V}_i)\simeq \ad^0(\pi_1^*\mb{H}_i).
    \end{equation}
\end{itemize}

    Now let $K=\mb{Q}(\zeta +\zeta^{-1})\subset \mb{C}$ be the trace field of $\mb{V}_i$, and $K' \subset K$ the adjoint trace field. Note that $[K: K']$ is a power of $2$, since $\tr(\ad(\gamma))$ is the square of $\tr(\gamma)$ for $\gamma \in \mathrm{SL}_2$. As we are assuming that $[K:\mb{Q}]$ is odd, this implies $K=K'.$  
    
    By \eqref{eqn:ad iso} and \cref{lemma:comminvariance}, $K$ is equal to the trace field of $\ad^0(\mb{H}_i)$. 

    It follows straightforwardly from the rigidity of hypergeometric connections and \cref{prop:hyp local mono} that there are only finitely many rank-two hypergeometric local systems with adjoint trace field $K$; therefore, by the pigeonhole principle, we may assume that all the~$\mb{H}_i$'s are obtained as restrictions of the same hypergeometric local system. Abusing notation,  we now denote each $\mb{H}_i$ by the same symbol $\mb{H}$. 
    
    Let $\{\tau_j\}\in \Aut(\mb{C})$ be a set of automorphisms whose restriction to $K$ give the set of embeddings of $K$ into $\mb{C}$; for any complex local system $\mb{L}$ and $\tau\in \Aut(\mb{C})$, we write~$\tau(\mb{L})$ for the Galois conjugate local system. 

    Summing over Galois conjugates, we deduce from \eqref{eqn:ad iso} that there is an isomorphism of local systems 
    \begin{equation}\label{eqn:ad iso total}
\bigoplus_{j}\ad^0(\pi_2^*\tau_j(\mb{V}_i))\simeq \bigoplus_j \ad^0(\pi_1^*\tau_j(\mb{H})).
    \end{equation}

As $\mb{V}$ comes from  a family of abelian varieties, $\bigoplus_j \tau_j(\mb{V})$ underlies a polarized $\mb{Z}$-variation of Hodge structures ($\mb{Z}$-PVHS). This induces a $\mb{Z}$-PVHS structure on the left-hand side \eqref{eqn:ad iso total}, which in turn induces one on the right-hand side: note that this latter $\mb{Z}$-PVHS structure may depend on $i$. We now use these integral structures in the case $i=1$ to define corresponding period maps; let us write $W$ for~$W_1$. Replacing  $W$ by an \'etale cover if necessary (since the  $\mb{Z}$-structures on these local systems may not agree on the nose), we may assume that the $\mb{Z}$-PVHS on~$\bigoplus_j \ad^0(\pi_1^*\tau_j(\mb{H}))$ is pulled back from~$W$.

We deduce that there exists a period domain $\Gamma\backslash \mc{D}^{+}$, with its tautological  $\mb{Z}$-PVHS $\mb{V}_{\Gamma \backslash \mc{D}^{+}}$ (see for example \cite[\S 3.3]{klingler2017hodge}, where $\Gamma \backslash \mc{D}^{+}$ is referred to as a connected Hodge variety), and period maps 
\[
i_{\mb{V}}: \Ctil \rightarrow \Gamma \backslash \mc{D}^{+}, \quad \ i_{\mb{H}}: W \rightarrow \Gamma \backslash \mc{D}^{+}, 
\]
such that $i_{\mb{V}}^*\mb{V}_{\Gamma \backslash \mc{D}^{+}}$ and $i_{\mb{H}}^*\mb{V}_{\Gamma \backslash \mc{D}^{+}}$ are isomorphic to 
\[
\bigoplus \tau_j(\ad^0\mb{V}) \quad\text{and}\quad \bigoplus \tau_j(\ad^0\mb{\mb{H}}),
\]
respectively. We may assume that $i_{\mb{V}}(\Ctil)$ is Mumford--Tate generic in~$\Gamma\backslash \mc{D}^+$

In fact, this period domain is a Shimura variety of abelian type (more precisely, a Hilbert modular variety) attached to $\Res^{K}_{\mb{Q}}\SL_2$, up to center, as $\ad^0(\pi_2^*\mb{V})$ has monodromy group $\PGL_2$ and trace field $K$;  therefore $\Gamma \backslash \mc{D}^{+}$ has dimension $[K:\mb{Q}]$.

For any $p\in \widetilde{\mc{M}}$, as $\mb{V}|_{\Ctil_p}$ has monodromy $\SL_2$ and trace field $K$, it follows that~$i_{\mb{V}}(\Ctil_p)$ is also Mumford--Tate generic. Using the correspondence between $\Ctil_{p_1}$ and~$W_1$, we deduce that the same holds for $i_{\mb{H}}(W)$.

Fix $w\in W$ such that the Hodge structure  $\bigoplus_j \tau_j(\ad^0 \mb{H}_{w})$ is Mumford--Tate generic in $\Gamma \backslash \mc{D}^+$. We may assume that $w\in W_i$ for all $i$, since $\cap W_i$ is obtained from~$W$ by removing countably many points.

Then \eqref{eqn:ad iso total} implies that, for each $p_i$, there is a point $q_i\in \pi_2(\pi_1^{-1}(w)) \subset \Ctil_{p_i}$ such that there is an isomorphism of $\mb{Q}$-Hodge structures
\[
\bigoplus_j \ad^0(\tau_j(\mb{V}_{i, q_i})) \simeq \bigoplus_j \ad^0(\tau_j(\mb{H}_{w})).
\]
By \cref{prop: iso-hs-generalized-hecke} this implies that, for each $i$, the point $i_{\mb{V}}(q_i)$ is in the generalized Hecke orbit of $i_{\mb{H}}(w)$; note that $i_{\mb{V}}(q_i)$ is also Mumford--Tate generic in $\Gamma \backslash \mc{D}^{+}$.

Let $Z\subset \Gamma \backslash \mc{D}^{+}$ 
be the Zariski closure of $i_{\mb{V}}(q_1), i_{\mb{V}}(q_2), \dots$. Applying \cref{thm:apz}, we deduce that  $Z$
is a finite union of weakly special subvarieties. On the other hand, each $i_{\mb{V}}(q_i)$ is Mumford--Tate generic, which implies that $Z=\Gamma \backslash \mc{D}^{+}$. This is a contradiction since $\dim Z\leq \dim \Ctil =2$. We have therefore shown the existence of such a $S_{\exc}$. 

Now fix an infinite sequence of points $p_1, p_2, \ldots \in \Mtil(\Qbar)\setminus S_{\exc}$, and as before write $\mb{V}_i$ for $\mb{V}|_{\Ctil_{p_i}}$.
 For  $i, j$, we say that $\mb{V}_i$ and $\mb{V}_j$ are equivalent if they are isomorphic after applying an automorphism of $\mb{P}^1$ and tensoring with a rank-one local system; this is an equivalence relation on the set $\{\mb{V}_{i}\}$. The same argument using \cref{thm:apz} shows that there must be infinitely many equivalence classes, and we may conclude. 
\end{proof}

\bibliography{siegel}
\bibliographystyle{alpha}

\end{document}